\documentclass[10pt, a4paper]{article} 

\author{Stefano Scrobogna \thanks{This research is supported by the Basque Government through the BERC 2018-2021 program and by Spanish Ministry of Economy and Competitiveness MINECO through BCAM Severo Ochoa excellence accreditation SEV-2017-0718 and through project MTM2017-82184-R funded by (AEI/FEDER, UE) and acronym "DESFLU".}}

\title{Zero limit of entropic relaxation time for  the Shliomis model of ferrofluids}

\usepackage[a4paper]{geometry}
\usepackage{empheq}
\usepackage{fourier}
\usepackage[T1]{fontenc}
\usepackage{amssymb, amsmath,  bm, mathrsfs}
\usepackage{amsfonts}
\usepackage[english]{babel}
\usepackage{amssymb,amsthm}
\usepackage{mathtools}
\usepackage{enumerate}
\DeclareMathAlphabet{\mathcal}{OMS}{cmsy}{m}{n}
\usepackage{hyperref}
\usepackage{cite}
\allowdisplaybreaks[1]
\usepackage[bbgreekl]{mathbbol}
\DeclareSymbolFontAlphabet{\mathbb}{AMSb}
\DeclareSymbolFontAlphabet{\mathbbl}{bbold}
\usepackage{xfrac}
\usepackage{enumerate}
\usepackage{color}

\renewcommand{\d}{\textnormal{d}}
\renewcommand{\div}{\textnormal{div}}
\newcommand{\fine}{\hfill$\blacklozenge$}
\newcommand{\curl}{\textnormal{curl}}
\newcommand{\pare}[1]{\left( #1 \right)}

\newcommand{\norm}[1]{\left\| #1 \right\|}
\newcommand{\av}[1]{\left| #1 \right|}
\newcommand{\bra}[1]{\left[ #1 \right]}
\newcommand{\set}[1]{\left\{ #1 \right\}}
\newcommand{\Hud}{\dot{H}^{ \frac{1}{2}}}
\newcommand{\Hmud}{\dot{H}^{ -\frac{1}{2}}}

\newcommand{\Hu}{\dot{H}^{ 1}}

\newcommand{\cP}{\mathcal{P}}
\newcommand{\LqHu}{L^4_{T} \dot{H}^1}

\newcommand{\bu}{\bar{u}}

\newcommand{\cA}{\mathcal{A}}
\newcommand{\cC}{\mathcal{C}}
\newcommand{\cL}{\mathcal{L}}

\newcommand{\cQ}{\mathcal{Q}}
\newcommand{\cF}{\mathcal{F}}

\newcommand{\cD}{\mathcal{D}}
\newcommand{\cE}{\mathcal{E}}
\newcommand{\cS}{\mathcal{S}}

\newcommand{\bR}{\mathbb{R}}
\newcommand{\bN}{\mathbb{N}}
\newcommand{\bM}{\mathbb{M}}
\newcommand{\cG}{\mathcal{G}}

\newcommand{\cO}{\mathcal{O}}
\newcommand{\cN}{\mathcal{N}}
\newcommand{\cT}{\mathcal{T}}

\newcommand{\Hs}{\dot{H}^s}
\newcommand{\ns}{\textnormal{NS}}

\makeatletter
\newcommand*{\RN}[1]{\expandafter\@slowromancap\romannumeral #1@}
\makeatother

\newcommand{\ps}[2]{\pare{ \left. #1 \ \right| \ #2 }}
\newcommand{\psc}[2]{\left\langle \left. #1 \ \right| \ #2 \right\rangle}

\newcommand{\hra}{\hookrightarrow}

\newcommand{\loc}{\textnormal{loc}}
\newcommand{\NS}{Navier-Stokes }

\theoremstyle{theorem}
\newtheorem{theorem}{Theorem}[section]
\newtheorem*{theorem*}{Theorem}
\newtheorem{prop}[theorem]{Proposition}
\newtheorem{lemma}[theorem]{Lemma}

\theoremstyle{definition}
\newtheorem{definition}[theorem]{Definition}
\newtheorem{rem}[theorem]{Remark}

\numberwithin{equation}{section}

\begin{document}

\AtEndDocument{\bigskip{\footnotesize
  \textsc{BCAM - Basque Center for Applied Mathematics,Mazarredo, 14,  E48009 Bilbao, Basque Country -- Spain} \par
  \textit{E-mail address:}  \texttt{\href{mailto:sscrobogna@bcamath.org}{sscrobogna@bcamath.org}}}}

 \maketitle

\begin{abstract}
We construct solutions for the Shilomis model of ferrofluids in a critical space, uniformly in the entropic relaxation time $ \tau \in\pare{0, \tau_0} $. This allows us to study the convergence when $ \tau\to 0 $ for such solutions. 
\end{abstract}

\tableofcontents

 \section{Introduction}

Ferrofluids are among the wide variety of synthetic materials created in the twentieth century. A ferrofluid  is a liquid that presents ferromagnetic properties, i.e. it becomes strongly magnetizable in the presence of an external magnetic field. Such a material does not exist naturally in the environment but it was created in 1963 by NASA \cite{stephen1965low} with a very specific goal: to be used as a fuel for rockets in an environment without gravity, hence the necessity to be pumped by applying a magnetic field. \\

Ferrofluids are collidal (a mixture in which one substance of microscopically dispersed insoluble particles is suspended throughout another substance) made of nanoscale ferromagnetic particles of a compound containing iron, suspended in a fluid. They are magnetically soft, which means that they do not retain magnetization once there is no external magnetic field acting on them. \\

 The versatility of such material and its peculiar property of being controlled via a magnetic field made it suitable to be  used in a whole variety of applications: ferrofluids are for instance used in loudspeakers in order to cool the coil and damp the cone \cite{Miwa2003}, as seals in magnetic hard-drives \cite{raj1982ferrofluid}, in order to reduce friction \cite{Huang2011} or enhance heat transfer \cite{LAJVARDI20103508, Sheikholeslami2015}. We refer the interested reader to \cite{Zahn2001}, the introduction of \cite{NST2016} and references therein for a survey of potential applications of ferrofluids. \\

 
 There are two systems of partial differential equations which are generally accepted as models for the motion of ferrofluids, which are known under the name of their developer, the Shliomis model \cite{shliomis1975non} and the Rosensweig model \cite{Rosensweig}. The mathematical analysis of such systems is very recent, in \cite{AH_Shilomis_weak, AH_Shilomis_strong, AH_Rosensweing_strong} and \cite{AH_Rosensweing_weak} it is proved that both Shliomis and Rosensweig model admit global weak and local strong solutions in bounded, smooth subdomains of $ \bR^3 $.  The same authors then considered as well thermal and electrical conductivity as well as steady-state solutions of various ferrofluids systems in  \cite{AH12-2, AH12, AH13, AH14, AH15, AH16} and \cite{HHl16}. In \cite{Scrobo_FF2D} and \cite{DS18} it was proved that the Rosensweig system for ferrofluids is globally well posed in dimension two.   \\

In the present work we consider the Bloch-Torrey regularization of the Shliomis system for ferrofluids in the whole three-dimensional space $ \bR^3 $
\begin{equation}\label{eq:Shilomis1} \tag{S1}
 		\left\lbrace
 			\begin{aligned}
 				& \rho_0 \pare{\partial_t u + \pare{u\cdot\nabla} u}- \nu \Delta u + \nabla p = \mu_0 \pare{M\cdot \nabla}H +\frac{\mu_0}{2}\ \curl\pare{M\times H}, && \pare{x, t}\in \bR^3\times \bR_+ ,\\
 				& \partial_t M + \pare{u\cdot\nabla} M -\sigma \Delta M = \frac{1}{2}\ \pare{ \curl\  u}\times M-\frac{1}{\tau}\pare{M-\chi_0 H} -\beta \ M \times\pare{M\times H}, && \pare{x, t}\in \bR^3\times \bR_+ ,\\
 				& \div \pare{H + M}=F , && \pare{x, t}\in \bR^3\times \bR_+ ,\\
 				& \div \ u=0, \ \curl \ H=0 , && \pare{x, t}\in \bR^3\times \bR_+ ,
 			\end{aligned}
 		\right.
 	\end{equation}
proposed by M.~Shliomis in \cite{Shliomis2, Shliomis3}. The function $ u $ represents the linear velocity of the fluid. If we denote as $ H_{\textnormal{ext}} $ the external magnetic field acting on the fluid $ F = -\div \ H_{\textnormal{ext}}  $ will be denoted as the \textit{external magnetic force}. The external magnetic field $ H_{\textnormal{ext}} $ induces a demagnetising field $ H $ and a magnetic induction $ B= H+M $. \\

 The parameter $ \sigma > 0 $ comes in play  when the diffusion of the spin magnetic moment is not negligible, we refer the reader to \cite{GaspariBloch}, and indeed it has a regularizing effect since in such regime the system \eqref{eq:Shilomis1} is purely parabolic. The constant $ \rho_0, \nu, \mu_0, \sigma, \tau, \chi_0, \beta $ are positive constants with a physical meaning. For the sake of readability we will consider the following normalization
\begin{equation*}
 \rho_0= \mu_0 = \beta =1. 
\end{equation*}
This assumption is made in order to simplify the readability of the paper only, and does not entails qualitative changes in the behavior of the solutions of \eqref{eq:Shilomis1} . 
On the other hand we will consider 
\begin{equation*}
\nu, \chi_0,  \sigma, \tau >0. 
\end{equation*}
We already mentioned why we consider $ \sigma > 0 $, while being $ \nu $ the kinematic viscosity of a fluid it is natural to assume it strictly positive. Let us hence now focus our attention on the remaining two physical parameters: $ \tau $ and $ \chi_0 $. The main scope of the present paper is in fact to describe the limit regimes of the solutions of \eqref{eq:Shilomis1} when $ \tau $ and $ \chi_0 $ tend to zero.  \\

\begin{itemize}

\item[$ \bf \tau $ :]
The parameter $ \tau $ is called the \textit{entropic relaxation time} of the system \eqref{eq:Shilomis1}, and roughly speaking it describes the average time required by the system \eqref{eq:Shilomis1} to recover a situation of equilibrium once it is perturbed. The average relaxation time of commercial grade ferrofluids is of the order
\begin{equation*}
\tau\approx 10^{-9} \ , 
\end{equation*}
whence, considering the smallness of such factor, it is reasonable to ask what happens to the solutions of \eqref{eq:Shilomis1} when $ \tau \to 0 $. 
 Despite the number of works on ferrofluids systems mentioned above  there is though, to the best of our knowledge, no systematic understanding of what this state of equilibrium might look like. On a formal level when $ \tau $ is very small the dynamic of the term
 \begin{equation*}
 \frac{1}{\tau}\pare{M-\chi_0 H}, 
 \end{equation*}
 is predominant in the evolution of $ M $, whence what is generally done in the literature is to consider the approximation
 \begin{equation}\label{eq:MH_balance_intro}
 M\approx \chi_0 H , 
 \end{equation}
 which, if satisfied, compensates the magnitude of $ \frac{1}{\tau}\pare{M-\chi_0 H} $. The main goal of the present work is hence to provide a first rigorous description of the solutions of \eqref{eq:Shilomis1} in the limit regime $ \tau\to 0 $, and to understand how and in which way small values of $  \tau $ can have stabilizing effects on the solutions of \eqref{eq:Shilomis1}. In a nutshell, we prove that when $ \tau \to 0 $
 \begin{equation}\label{eq:conv_MH_intro}
 \pare{M, H} \xrightarrow{\tau\to 0} \pare{\chi_0 G_F, \ G_F}, 
 \end{equation}
 where $ G_F $ is a function depending upon the external magnetic field only, while
 \begin{equation*}
 u\xrightarrow{\tau\to 0} U, 
 \end{equation*}
 where $ U $ is the unique solution of the following \NS system with hydrostatic-magnetic pressure
 \begin{equation*}
 \left\lbrace
 \begin{aligned}
 & \partial_t U + U \cdot \nabla U -\nu \Delta U = -\nabla \pare{\pi - \chi_0 P_F} , \\
 & \div\ U=0, 
 \end{aligned}
 \right.
 \end{equation*}
 where $ P_F $ depends only  on the external magnetic force $ F $, and in particular assumes the following explicit form
 \begin{equation*}
 P_F = \frac{1}{ 2 \pare{1+\chi_0}^2 }  \ \nabla \av{\nabla\Delta^{-1} F}^2. 
 \end{equation*}
 The derivation of such magnetic pressure is somewhat surprising and it will be discussed in detail later in the manuscript.

 \item[$ \chi_0 $ :] The dimensionless parameter $ \chi_0 $ is called \textit{magnetic susceptibility} and indicates whether a material is attracted into or repelled out of a magnetic field. If the magnetic susceptibility is greater than zero, the substance is said to be "paramagnetic"; the magnetization of the substance is higher than that of empty space. If the magnetic susceptibility is less than zero, the substance is "diamagnetic"; it tends to exclude a magnetic field from its interior. Since ferrofluids are magnetically soft materials their magnetization is  higher than that of the vacuum, hence the motivation that lead us to suppose $ \chi_0 >0 $. Experimental results show that for oil-based colloidals $ \chi_0\in\bra{0.3 \ , \ 4.3} $, while for water-based colloidals $ 0<\chi_0\ll 1 $: water-based ferrofluids are hence almost neutral to external magnetic forces.

 \end{itemize}

 The results provided and quickly illustrated here above formally justify the physical intuition of how the parameters $ \tau $ and $ \chi_0 $ influence the dynamics of \eqref{eq:Shilomis1}. Rigorously proving such results at a mathematical level is though not so immediate. The singular linear perturbation
 \begin{equation*}
  \frac{1}{\tau}\pare{M-\chi_0 H}, 
 \end{equation*}
 which is reminiscent of singular perturbations arising in problems in geophysical fluid mechanics (cf. \cite{LM98, monographrotating, gallagher_schochet} etc) is in fact of a different nature; it has no definite sign and more importantly it depends upon the external magnetic field $ F $. Being this the case the singular perturbation $  \frac{1}{\tau}\pare{M-\chi_0 H} $ does not supply a zero $ L^2 $ energy contribution as it happens for rotating fluids (\cite{CDGG2, gallagher_schochet}), compressible fluids (\cite{LM98, DanchinMach, Scrobo_Ngo}) or stratified fluids (\cite{charve1, charve2, Scrobo_Froude_FS, Scrobo_Froude_periodic, Sang_Scrobo_Froude}), whence it is not possible to construct global weak or local strong solutions uniformly in $ \tau > 0 $ by means of energy methods as it is done in the examples mentioned above. \\
 
 The way hence to construct a sequence $ \pare{U^\tau}_{\tau\in\pare{0, \tau_0}} $ of solutions of \eqref{eq:Shilomis1} passes through the understanding of the physical properties of the singular perturbation  $ \frac{1}{\tau}\pare{M-\chi_0 H} $; in the geophysical fluid dynamics setting mentioned above typically the singular perturbation induces hi-frequency oscillations on which it is possible to prove dispersive estimates. In the present case the singular perturbation seems to produce a \textit{damping} effect, but it is not at all clear how such damping acts on the system; the singular perturbation has in fact no definite sign in the unknowns $ u, M, H $ and hence we cannot immediately conclude in this way. \\
 
 The problem is that the unknowns $ M $ and $ H $ are not suitable in order to describe the system \eqref{eq:Shilomis1} \textit{uniformly in} $ \tau $. One part of the unknown is in fact effectively damped to zero while the other converges toward a stationary state; we must hence find another set of unknowns which somehow explicit such problem. If we define
 \begin{align*}
 \cP = 1_{\bR^3} - \Delta^{-1}\nabla\div , && \cQ = \Delta^{-1}\nabla\div, 
 \end{align*}
 it is rather easy to deduce from the magnetostatic equation $ \div\pare{M+H}=F $ that\footnote{Here we use the fact that $ \curl \ H=0 $. }
 \begin{equation}\label{eq:intro:MH_relation}
 H = -\cQ M +\Delta^{-1}\nabla F .
 \end{equation}
Using the relation \eqref{eq:intro:MH_relation} we can re-write the singular perturbation  $ \frac{1}{\tau}\pare{M-\chi_0 H} $ as 
\begin{equation}\label{eq:singular_pert_tau}
\frac{1}{\tau}\pare{M-\chi_0 H} = \frac{1}{\tau} \pare{1+\chi_0\cQ} M -\frac{\chi_0}{\tau}\Delta^{-1}\nabla F.
\end{equation}
This singular perturbation presents two immediate characteristics which are not present in some classical works on singular perturbation problems (\cite{GS3, GallagherSaint-Raymondinhomogeneousrotating, gallagher_schochet, CDGG2, Scrobo_primitive_horizontal_viscosity_periodic}, the list is far from being exhaustive);
\begin{itemize}

\item if the magnetic susceptibility $ \chi_0 $ is large, which is the case for oil-based ferrofluids as explained above,  the operator $ \pare{1+\chi_0\cQ} $ has not positive sign, 

\item the singular perturbation \eqref{eq:singular_pert_tau} is \textit{linear and non-homogeneous}, case that, to the best of our knowledge, has not yet been treated in the literature. 

\end{itemize}
 
 \noindent
 Instead we decide to tailor a specific approach to the problem; applying the operator $ \cP $ to the evolution equation of $ M $, and hence as well  to the singular perturbation $ \frac{1}{\tau}\pare{M-\chi_0 H} $, we deduce that
 \begin{equation*}
  \frac{1}{\tau}\cP\pare{M-\chi_0 H} =  \frac{1}{\tau}\cP M. 
 \end{equation*}
 It is hence clear that $ \cP M $, the divergence-free part of $ M $, is damped to zero in the evolution of the system \eqref{eq:Shilomis1}. The next very natural step is to compute the second term of the Hodge decomposition of $ \frac{1}{\tau}\pare{M-\chi_0 H} $ which is 
 \begin{equation}\label{eq:intro:1tauQM}
 \frac{1}{\tau}\cQ\pare{M-\chi_0 H} = \frac{1+\chi_0}{\tau}\pare{\cQ M - \frac{\chi_0}{1+\chi_0}\nabla\Delta^{-1} F }.
 \end{equation}
 In such setting we can hence deduce the new limit $ \tau\to 0 $ formal balance
 \begin{equation*}\label{eq:QM_balance_intro}
 \cQ M \approx \frac{\chi_0}{1+\chi_0}\nabla\Delta^{-1} F, 
 \end{equation*}
 which is much better than the balance \eqref{eq:MH_balance_intro} since now we obtain an asymptotic which depends only on the external magnetic field $ F $ and \textit{not on another unknown}. We can as well recover the formal limit asymptotic for $ H $ as well from the relation \eqref{eq:intro:MH_relation}. \\
 
 Despite a better understanding of the asymptotics as $ \tau\to 0 $ we did not yet solve the main problem of the mathematical construction of solutions uniformly in $ \tau $, the singular perturbation on the r.h.s. of \eqref{eq:intro:1tauQM} has still sign not defined, and appears in the system \eqref{eq:Shilomis1} applying the operator $ \cQ $ to the evolution equation of $ M $, i.e.
 \begin{equation}\label{eq:intro:equazione_QM}
 \partial_t \cQ M -\sigma\Delta \cQ M + \frac{1+\chi_0}{\tau}\pare{\cQ M - \frac{\chi_0}{1+\chi_0}\nabla\Delta^{-1} F } = \text{ Nonlinear terms }. 
 \end{equation}
 We remark at this point though that $ F $ is not an unknown of the problem. We can hence subtract $ \frac{\chi_0}{1+\chi_0}\pare{\big. \partial_t -\sigma \Delta}\nabla\Delta^{-1} F $ from both sides of \eqref{eq:intro:equazione_QM} and defining the new unknown $ r=\cQ M - \frac{\chi_0}{1+\chi_0}\nabla\Delta^{-1} F  $ we can deduce the evolution equation for $ r $
 \begin{equation*}
 \partial_t r -\sigma\Delta r + \frac{1+\chi_0}{\tau} \ r = \underbrace{ - \frac{\chi_0}{1+\chi_0}\pare{\big. \partial_t -\sigma \Delta}\nabla\Delta^{-1} F}_{\text{Outer force } f} + \text{ Nonlinear terms }, 
 \end{equation*}
 which is now damped and diffused, and we can close our argument. In detail, the new evolutionary system so obtained is of the form (here $ m=\cP M $)
 \begin{equation}\label{eq:intro:S2}
 \begin{aligned}
 & \partial_t u -\nu \Delta u &&    &&= \text{ Nonlinear terms }, \\
 & \partial_t m -\sigma \Delta m && +\frac{1}{\tau}\ m &&= \text{ Nonlinear terms }, \\
  & \partial_t r -\sigma\Delta r &&  + \frac{1+\chi_0}{\tau} \ r  && =  \text{ Nonlinear terms }&& +f. 
 \end{aligned}
 \end{equation}

 \noindent
 At this point we hence expect the unknown $ m, r $ in \eqref{eq:intro:S2} to be exponentially damped to zero at a rate $ \cO \pare{ e^{-t/\tau}} $. There are though two immediate obstructions to such deduction:
 \begin{itemize}
 
 \item The external force $ f $ in the evolution equation is an $ \cO\pare{1} $ function, 
 
 \item There are terms on the r.h.s. of \eqref{eq:intro:S2} which are $ \cO\pare{1} $ functions for $ m, r\to 0 $. 
 
 \end{itemize}
 
 Whence despite the tendency of the evolution of $ m $ and $ r $ is to be quickly damped to zero there are external forces in the system \eqref{eq:intro:S2} which are genuinely bigger than $ \tau $ and which induce a higher order growth on the unknowns $ m $ and $ r $. It is in this context that slightly more involved parabolic estimates are required (see Lemma \ref{lem:linear_damping_estimate} for the exact estimates used in this work) in order to see that $ {m, r} \xrightarrow{\tau\to 0} 0 $. A downside of such approach is that we are not able to quantify the rate of convergence \label{note:quantification} of $ m, r $ to zero as $ \tau\to 0 $, due indeed to the perturbative effects induced by the $ \cO\pare{1} $ perturbations. \\

Let us finally mention an unexpected stabilizing effect we remarked. We already mentioned and explained in reasonable detail that the components $ m $ and $ r $ of \eqref{eq:intro:S2} are subjected to a damping-in-time. Let us hence now consider that we we want to construct $ \LqHu $ solutions of \eqref{eq:intro:S2} in a fashion very similar  to what is done for the more familiar incompressible \NS equations. It is clear hence that if $ \tau $ is sufficiently small, hence the damping coefficient is very large, for any $ t > 0 $ the functions $ m\pare{t}, r\pare{t} $ are drawn to zero rather vigorously so that we expect that they are "small". This crude intuition lead us to think that we might as well expect to construct global solutions for \eqref{eq:intro:S2} imposing a smallness hypothesis on $ u_0 $, the initial data of the velocity flow, and $ \tau $: we can in fact construct global solutions substituting a smallness hypothesis on $ m_0, r_0 $ with a smallness hypothesis on $ \tau $. Such result is attainable only if we construct solutions in the critical space $ \LqHu $ and not in, say, $ L^\infty_T \Hud\cap L^2_T \dot{H}^{\frac{3}{2}} $; the damping effect has no effects on the $ L^\infty_T \Hud $ norm.

\subsection{Results and organization of the paper}

The main goal of the present paper is to study the properties of the solutions of system \eqref{eq:Shilomis1}  \textit{when the parameter $ \tau $ is small} or converging to zero, indeed hence the first (and main) result of the present work is an existence result which is uniform for $ \tau $ belonging to a suitable right-neighborhood of zero, whose size depends on the magnitude of the initial data. \\

From now on given a Banach space $ X $, any $ T\in\bra{0, \infty}, \ k\in\bN $ and $ p\in\bra{1, \infty} $ we denote as $ W^{k, p}_T X $ the space $ W^{k, p}\pare{\left[ 0, T \right)\big.  ; X} $. Given any Sobolev or Lebesgue space if the domain is not specified it is implicitly assumed to be $ \bR^3 $. Given any $ s<3/2 $ we define  the \textit{homogeneous}  Sobolev space $ \dot{H}^s \pare{\bR^d} $ as the closure of $ \cS_0\pare{\bR^d} $ with respect to the norm
\begin{equation*}
\norm{v}_{\dot{H}^s \pare{\bR^d}} = \pare{\int_{\bR^d}\av{\xi}^{2s}\av{\hat{u}\pare{\xi}}^2\d \xi}^{1/2},
\end{equation*}
while for any $ s\in\bR $ the \textit{non-homogeneous} Sobolev space $ H^s\pare{\bR^d} $ is composed of the tempered distributions $ v $ such that $ \pare{1+\sqrt{-\Delta} \ }^s \ v \in L^2\pare{\bR^d} $. Given any $ k\in\bN $ and $ p \in \bra{1, \infty} $ we say that $ v\in \dot{ W}^{k, p}_T X $ if $ \partial_t^k v \in L^p_T X $ and  $ v\in {W}^{k, p}_T X $ if $ \pare{ 1+\partial_t^k} v \in L^p_T X $. Given a vector field $ V : \bR^n\times \bR_+ \to \bR^m $ we will write $ V\in {W}^{k, p}\pare{\bra{0, T}; \Hs\pare{\bR^n}} $ instead than writing  $ V\in \pare{ {W}^{k, p}\pare{\bra{0, T}; \Hs\pare{\bR^n}}}^m $ in order to simplify the overall notation. 
  The capital letter $ C $ will always indicate a positive value independent by any parameter of the problem whose value may implicitly vary from line to line while $ c=\min\set{\big. \nu, \sigma} $. \\

Let us  moreover suppose the external magnetic field $ F $ belongs to the space\footnote{Remark that in this case the Sobolev space $ H^2 $  is considered to be non-homogeneous. }
\begin{equation} \label{eq:regularity_F}
F\in L^4_{\loc}\pare{ \bR_+;  L^2}\cap W^{1, 2}_{\loc} \pare{\bR_+;  H^2}.
\end{equation}
 We underline that the external magnetic field is \textit{not} an unknown of the problem, hence it is in no way restrictive to assume that it is smooth and integrable.

\begin{theorem}\label{thm:main_thm}
Let $ u_0\in \Hud, F\in L^4_{\loc}\pare{ \bR_+;  L^2}\cap W^{1, 2}_{\loc} \pare{\bR_+;  H^2} $. There exists a $ \rho, \varrho_0>0 $, where $ \rho > 2\varrho_0 $, and a $ T=T_{\varrho_0} > 0 $ defined as
\begin{equation}\label{eq:def_T}
T=T_{\varrho_0} = \sup \set{t \geqslant 0\Big.  \ \left| \ \norm{F}_{L^4\pare{\bra{0, t};L^2 \pare{\bR^3} }}<\varrho _0 \textnormal{ and } F\in W^{1, 2} \pare{ \left[0, t\right) ;  H^2 }\right. },
\end{equation}
sufficiently small so that
\begin{equation*}
\norm{F}_{L^4_TL^2}\leqslant \varrho_0 \leqslant \frac{\min\set{\min\set{\big.\nu, \sigma}^{1/2}, \ \min\set{\big.\nu, \sigma}^{3/4} }}{C}, 
\end{equation*}
 and such that if we define
\begin{align*}
 m_0 = \pare{1- \Delta^{-1}\nabla\div} M_0, \hspace{5mm} r_0= \Delta^{-1}\nabla\div \ M_0 -\frac{\chi_0}{1-\chi_0} \ \nabla\Delta^{-1}F. 
\end{align*}

\begin{enumerate}[{a)}]

\item \label{enum:thm1} Let $ u_0, m_0, r_0 \in \Hud $ be such that
\begin{align*}
\norm{u_0}_{\Hud}\leqslant \frac{\nu^{1/4}}{C} \ \rho,
&&
\norm{\pare{m_0, r_0}}_{\Hud}\leqslant\frac{\sigma^{1/4}}{C}\ \rho,  
\end{align*}
and
\begin{equation}\label{eq:smallness_tau1_thm}
\tau < \frac{\pare{1+\chi_0}^{7/3}}{ C\ \chi_0^{4/3}} \pare{\norm{ F}_{L^2_T\dot{H}^{2}}   + \norm{F}_{\dot{W}^{1, 2}_T L^2}}^{-4/3} \ \varrho_0^{4/3}.
\end{equation}
Then there exist a unique solution $ \pare{u, M, H} $ of \eqref{eq:Shilomis1} with initial data $ \pare{u_0, M_0} $ in  the space $ \cC_T \Hud \cap \LqHu $. 

\item \label{enum:thm2} Let $ {U_0} = \pare{u_0, m_0, r_0} \in \Hud $ arbitrarily large and $ \tau > 0 $ satisfy the relation \eqref{eq:smallness_tau1_thm}, there exist a $ T^\star = T^\star_{U_0}\in\pare{0, T} $, where $ T $ is defined in \eqref{eq:def_T},  such that the system \eqref{eq:Shilomis1} admits a unique solution with initial data $ \pare{u_0, M_0} $ in  the  space $ \cC_{T^\star} \Hud \cap  L^4_{T^\star}\dot{H}^1 $. 

\item \label{enum:thm3}  Let  $ u_0\in\Hud $  be such that
\begin{equation}\label{eq:smallness_vel_flow_thm}
\norm{u_0}_{ \Hud}\leqslant\frac{\nu^{1/4}}{C} \ \rho, 
\end{equation}
and $ m_0, r_0\in\dot{H}^1 $  arbitrary. Let $ \tau $ be sufficiently small so that
\begin{equation}\label{eq:smallness_tau_thm}
\tau\leqslant \min\set{
\frac{\rho^4}{C\pare{\norm{m_0}_{\dot{H}^1}^4 + \norm{r_0}_{\dot{H}^1}^4}}
\hspace{2mm}, \hspace{2mm}
 \frac{ \pare{{1+\chi_0}}^{7/3}\varrho_0^{4/3}}{C\chi_0^{4/3} \  \pare{\norm{ F}_{L^2_T\dot{H}^{2}}   + \norm{F}_{\dot{W}^{1, 2}_T L^2}}^{4/3}}
}. 
\end{equation}
Then there exist a unique solution $ \pare{u, M, H} $ of \eqref{eq:Shilomis1} with initial data $ \pare{u_0, M_0} $ in  the space $ \cC_T \Hud \cap \LqHu $.

\item \label{enum:thm4} Let $ u_0\in\Hud $ arbitrarily large and let $ \tau $ satisfy \eqref{eq:smallness_tau_thm}, there exist a $ T^\star\in\pare{0, T} $ such that the system \eqref{eq:Shilomis1} admits a unique solution with initial data $ \pare{u_0, M_0} $ in the space $\cC_{T^\star} \Hud \cap L^4_{T^\star}\dot{H}^1 $.

\end{enumerate}
\end{theorem}

\begin{rem}

\begin{itemize}
\item The value $ T $ defined in \eqref{eq:def_T} is well defined and strictly positive since the application
\begin{equation*}
t\mapsto \norm{F}_{L^4\pare{\bra{0, t}; L^2 \pare{\bR^3}}},
\end{equation*}
 is continuous and non-decreasing in $ \bR_+ $ and zero when $ t=0 $. 
From now on when we write $ T $ we will always consider the value defined by \eqref{eq:def_T}. Let us remark that if $ F $ is sufficiently small in $ L^4 \pare{ \bR_+;  L^2} $ then $ T $ can be equal to infinity as well, transforming hence the results stated in the points \ref{enum:thm1} and \ref{enum:thm3} in genuinely global-in-time results.

\item In the definition \eqref{eq:def_T} we must include the hypothesis $ F\in W^{1, 2}_T H^2 $ \textit{only} for the case in case in which $ T=\infty $. In fact a priori it may as well happen that $ \norm{F}_{L^4\pare{\bR_+; L^2}}\leqslant \varrho_0 $, $ F\in W^{1, 2}_{\loc} \pare{\bR_+ ; H^2 } $ but $ F $ \textit{does not belong} to the space $ W^{1, 2} \pare{\bR_+ ; H^2 }  $. In such setting we implicitly use the fact that $F\in  W^{1, 2} \pare{\bR_+ ; H^2 }  $ in setting the smallness hypothesis \eqref{eq:smallness_tau1_thm} and \eqref{eq:smallness_tau_thm}. 

\item The points \ref{enum:thm1} and \ref{enum:thm2} in the statement of Theorem \ref{thm:main_thm} can be rephrased as "global" existence for small data and "local" existence for arbitrary initial critical data. Indeed the point \ref{enum:thm1} is a proper global-in-time result only if $ T=\infty $ where $ T $ is defined in \eqref{eq:def_T}: the hypothesis on $ T $, which is a smallness hypothesis on the norm of $ F $, avoids that the external magnetic field $ F $ pumps too much energy in the system. It is in fact intuitive that, if $ M, H $ have to satisfy the magnetostatic equation
\begin{equation*}
\div \pare{M+H}=F, 
\end{equation*} 
and $ F $ is "arbitrarily large" then the curl-free part of $ M+H $ will be arbitrarily large as well (in some appropriate, non specified, topology). In such scenario $ M $ and $ H $ result to be hence "large" and it is not possible to construct solutions via a fixed point theorem around a stationary state of \eqref{eq:Shilomis1}. 

\item The points \ref{enum:thm3} and \ref{enum:thm4} are again a "global" and "local" existence result. We focus now on the characteristics of the point \ref{enum:thm3}. It is worth noticing  that we impose a smallness hypothesis on the initial data for the velocity field $ u_0 $ and for $ \tau $. We let hence $ M_0 $ and $ H_0 $ be \textit{arbitrarily large in} $ \dot{H}^1 $; this effect  is due to the term $ \frac{1}{\tau}\pare{M-\chi_0 H} $ in \eqref{eq:Shilomis1}. Roughly speaking such term provides a damping with damping coefficient $ \tau^{-1} $ which we will exploit in order to damp the $ \dot{H}^1 $ norm of $ M_0 $ and $ H_0 $ sufficiently fast so that the overall $ \LqHu $ norm will result to be small, hence to possibility  apply a fixed point theorem. It is also for this reason that we construct solutions in the critical space $ \LqHu $ instead that, say, the more natural critical energy space $ L^\infty_T \dot{H}^{\frac{1}{2}}\cap L^2_T \dot{H}^{\frac{3}{2}}  $. If we start with large $ \Hud $ data the damping effect does not influences the overall $ L^\infty_T \dot{H}^{\frac{1}{2}} $ norm of the solution, hence a fixed point theorem based on the smallness of the norm is not applicable in such setting \textit{when large initial data is considered}, in fact $ M_0 $ and $ H_0 $ can even be \textit{unbounded} in $ \Hud $, but they have to be finite in $ \dot{H}^1 $ in order to apply the result in Theorem \ref{thm:main_thm}, \ref{enum:thm3}. 

\item Let us remark again that in the point \ref{enum:thm3} of Theorem \ref{thm:main_thm} the only hypothesis assumes on $ m_0, r_0 $ is a smallness hypothesis with respect to $ \tau $ \textit{in the space $ \dot{H}^1 $}. The data  $ m_0, r_0  $ can even be \textit{unbounded} in the critical space $ \Hud $; we are hence able to construct a global-in-time solution for the system \eqref{eq:Shilomis1} imposing a smallness hypothesis on the initial velocity flow $ u_0 $ only. 

\item Since the points \ref{enum:thm3} and \ref{enum:thm4} represent an unexpected dynamical effect for the system \eqref{eq:Shilomis1} we will prove explicitly only the point \ref{enum:thm3}, being the other points simple variations of this one. 

\item Even if we restrain ourselves to the more familiar setting stated in the points \ref{enum:thm1} and \ref{enum:thm2} we construct solutions in the critical space $ \LqHu $ imposing initial data in $ \Hud $; we construct hence potentially \textit{infinity $ L^2 $ energy solutions} for \eqref{eq:Shilomis1}. This work is, to the best of our knowledge, the first work in which infinite $ L^2 $ energy solutions for ferrofluids systems are constructed. It is worth to remark that if we try to construct solutions for \eqref{eq:Shilomis1} using the natural $ L^2 $ energy of the system (see \cite{AH_Shilomis_weak, AH_Rosensweing_weak, Scrobo_FF2D, DS18}) uniformly in $ \tau $ we deduce an estimate of the form
\begin{equation*}
\cE \pare{t} + c_{\tau} \int_0^t \cD \pare{t'} \d t' \leqslant \frac{C}{\tau} ,
\end{equation*}
where $ \cE $ and $ \cD $ are the natural energy and dissipation of the system \eqref{eq:Shilomis1}. Energy methods are hence not applicable in order to construct solutions of \eqref{eq:Shilomis1} uniformly in $ \tau $ since the r.h.s. of the above equation blows-up as $ \tau\to 0 $ and does not provide uniform estimates. 

\end{itemize}
\end{rem}

Theorem \ref{thm:main_thm} is hence an existence result for solutions of \eqref{eq:Shilomis1} which holds \textit{uniformly for $ \tau $ in a right neighborhood $ \pare{0, \tau_0} $ of zero}. As we already explained in detail in the remark above the points \ref{enum:thm3} and \ref{enum:thm4} deal with \textit{stabilizing properties} of solutions of \eqref{eq:Shilomis1} when $ \tau $ is small. It is hence a natural question at this stage to ask whether solutions of \eqref{eq:Shilomis1} converges (and if they do, in which topology) to some limit flow. \\
\noindent
It turns out that the term $ \frac{1}{\tau}\pare{M-\chi_0 H} $ acts effectively as an exponential damping on the components $ M, H $; such damping effect is though not immediate to prove, and neither it is immediate to rigorously deduce from the structure of the equations \eqref{eq:Shilomis1}. The precise statement is the following one:

\begin{theorem}\label{thm:convergence}
Let us consider the same hypothesis as in Theorem \ref{thm:main_thm}, \ref{enum:thm3} and let us suppose moreover that $ m_0, r_0\in\Hud $, let us consider any (small) $ \varepsilon\in\pare{0, T} $, then
\begin{equation}\label{eq:convergence_MH}
\begin{aligned}
M \xrightarrow{\tau\to 0} \frac{\chi_0}{1+\chi_0} \nabla\Delta^{-1} F , && \text{in } L^\infty\pare{\pare{\varepsilon, T}; \Hud}, \\
H \xrightarrow{\tau\to 0} \frac{1}{1+\chi_0} \nabla\Delta^{-1} F , && \text{in } L^\infty\pare{\pare{\varepsilon, T}; \Hud}.
\end{aligned}
\end{equation}
Moreover the following convergence hold true
\begin{equation}\label{eq:convergence_space}
\begin{aligned}
u & \xrightarrow{\tau\to 0} \bu, & \text{in } & L^\infty\pare{\pare{\varepsilon, T}; \Hud}, \\
\nabla u & \xrightarrow{\tau\to 0} \nabla\bu, & \text{in } &  L^2\pare{\pare{\varepsilon, T}; \Hud},
\end{aligned}
\end{equation}
 where $\bu$ is the solution of the following  incompressible \NS system with additional magnetic pressure
\begin{equation}\label{eq:limit_system_thm}
\left\lbrace
\begin{aligned}
& \partial_t\bu + \bu\cdot\nabla\bu -\nu \Delta\bu +\nabla\bar{p} = \frac{\chi_0}{2 \pare{1+\chi_0}^2}\ \nabla\av{\nabla\Delta^{-1} F}^2 , \\
& \div\ \bu=0, \\
&\left.\bu\right|_{t=0} = u_0.  
\end{aligned}
\right.
\end{equation}

\end{theorem}

\begin{rem}
\begin{itemize}

\item We want to underline that the convergence mentioned in Theorem \ref{thm:convergence} takes place only in the topology \eqref{eq:convergence_space}; this is justified by the fact that when $ \tau\to 0 $ a genuine damping effect is induced, whence we \textit{cannot} have immediate convergence (i.e. in $ L^\infty\pare{\pare{0, \varepsilon}; \Hud} $) to the limit function.

\item Let us denote respectively with $ M^0 $ and $ H^0 $ the r.h.s. of \eqref{eq:convergence_MH}, i.e.

\begin{align*}
M^0 & = \frac{\chi_0}{1+\chi_0} \nabla \Delta^{-1}F, & 
H^0 & = \frac{1}{1+\chi_0} \nabla \Delta^{-1}F . 
\end{align*}
If we let $ \tau \to 0 $ in the equation for $ M $ appearing in \eqref{eq:Shilomis1} and we denote $ \bu = \lim _{\tau\to 0} u $ consistently with the notation of Theorem \ref{thm:convergence} it looks at a fist glance that such limiting process on the equation of $ M $ induces a nonlinear constraints which relates $ \bu $ with the limiting flows $ M^0, \ H^0 $ which are uniquely determined by the external magnetic force $ F $, whence $ \bu = \bu \pare{F} $ which could not satisfy \eqref{eq:limit_system_thm} in $ \pare{\varepsilon, T} $ making of the limit system an overdetermined problem. This is indeed not the case since despite the following convergence holds true
\begin{equation*}
M-\chi_0 H \xrightarrow{\tau\to 0} 0 ,
\end{equation*}
in a sufficiently weak sense (say $ \cD'\pare{\bR^3\times\pare{\varepsilon, T}} $ ) we are unable to quantify the rate of convergence toward zero of $ M-\chi_0 H $ as it has been already mentioned at page \pageref{note:quantification}. Whence we do not actually know to which element will the term 
\begin{equation*}
\frac{1}{\tau}\pare{M-\chi_0 H}, 
\end{equation*}
converge. This can though be easily deduced, at least in a formal way; let us consider a $ \phi \in \cD \pare{\bR^3\times\pare{\varepsilon, T}}  $, considering the convergences \eqref{eq:convergence_MH} and \eqref{eq:convergence_space}, and supposing there exists a $ f \in \cD' \pare{\bR^3\times\pare{\varepsilon, T}}  $ so that
\begin{equation*}
\frac{1}{\tau}\pare{M-\chi_0 H}\xrightarrow[\tau\to 0]{\cD'} - f, 
\end{equation*}
testing the equation \eqref{eq:Shilomis1} with $ \phi $ and letting $ \tau \to 0 $ the limit equation solved by $ M^0 $ (in $ \cD' $) is
\begin{equation*}
\partial_t M^0 + \pare{\bu\cdot\nabla} M^0 -\sigma \Delta M^0 - \frac{1}{2}\ \pare{ \curl\  \bu}\times M^0 = f , 
\end{equation*}
whence the limit problem is consistently expressed. 
\end{itemize}
 \fine
\end{rem}

The present paper is structured as follows:

\begin{itemize}

\item Section \ref{sec:preliminaries} is devoted to introduce some preliminary result which we use all along the paper. In particular Section \ref{sec:linear_parabolic_estimates} consists of a series of bounds for linear parabolic equations with damping which will result very important in the application of the fixed point theorem in Section \ref{sec:fixed_poin_application}. 

\item In Section \ref{sec:ref_syst} we define a new set of unknowns for the system \eqref{eq:Shilomis1} so that we can deduce a new system (see \eqref{eq:Shilomis2} for the detailed definition) which highlights and makes explicit the damping effect induced by the singular perturbation $ \tau^{-1}\pare{M-\chi_0 H} $. Such procedure has been already outlined in the introduction, in Section \ref{sec:ref_syst} we make this argument rigorous. 

\item Section \ref{sec:existence_sol} is the core of the present article, in such section we prove Theorem \ref{thm:main_thm} which is the most technical result of the present paper. The proof of Theorem \ref{thm:main_thm} consists in a fixed point argument, which has to be performed carefully, and more importantly, has to be adapted to highlight the particular properties of the system \eqref{eq:Shilomis1} (most notably the damping effects induced by the singular perturbation $ \tau^{-1}\pare{M-\chi_0 H} $). 

\item Section \ref{sec:conv_tau} is devoted to the proof of Theorem \ref{thm:convergence}. Using the result proved in Section  \ref{sec:existence_sol} (i.e. Theorem \ref{thm:main_thm}, an existence result uniform in $ \tau $) we prove at first that some part of the the system is effectively damped to zero in a critical norm away from $ t=0 $, next we use such convergence in order to prove that the velocity flow converges toward the system \ref{eq:limit_system_thm}.

\end{itemize}

\section{Preliminaries}\label{sec:preliminaries}

All along the present paper we will consider nonlinear interactions of (homogeneous) Sobolev functions. It is well known that, in a more general context, the product of two distributions in, a priori, not well defined, cf. \cite{Schwartz54}. In the context of Sobolev functions we can state the following elementary criterion:

\begin{lemma}\label{lem:Sob_product_rules}
Let $ \pare{s, t}\in\bR^2 $ and $ d\in\bN\setminus \set{0} $ be such that $ s, t < \frac{d}{2} $ and $ s+t >0 $. 
The point-wise product application maps continuously $ \Hs \pare{\bR^d}\times \dot{H}^t\pare{\bR^d} $ onto $ \dot{H}^{s+t-\frac{d}{2}}\pare{\bR^d} $, i.e. if we consider $ u \in \Hs \pare{\bR^d}, \ v\in \dot{H}^t\pare{\bR^d} $, there exists a $ C>0 $ depending only on the dimension $ d $ so that
\begin{equation*}
\norm{u \ v}_{\dot{H}^{s+t-\frac{d}{2}}\pare{\bR^d} }\leqslant C \norm{u}_{\Hs \pare{\bR^d}}\norm{v}_{\dot{H}^t\pare{\bR^d}}. 
\end{equation*}
\end{lemma}
\begin{rem}
There exists a non-homogeneous counterpart of Lemma \ref{lem:Sob_product_rules}. \fine
\end{rem}

Lemma \ref{lem:Sob_product_rules} belongs to the mathematical folklore, and can be stated as well for periodic vector fields, cf. \cite{gallagher_schochet}. Such result is widely used in the \NS theory and goes under the name of \textit{product rules for Sobolev spaces}. All along the paper we will use continuously, even implicitly, the result stated in Lemma \ref{lem:Sob_product_rules}.

\begin{definition}
Let $ X $ be an abstract Banach space and $ T_p : X^p\to X $ a $ p $--linear map onto $ X $. We define
\begin{equation*}
\norm{T_p} = \sup _{\phi_1, \ldots , \phi_p \in B_X\pare{0, 1}} T_p\pare{\phi_1, \ldots , \phi_p}. 
\end{equation*}
\end{definition}

\begin{prop}\label{prop:fixed_point}
Let $ X $ be a Banach space and let $ T_p : X^p \to X, \ p=1,2,3 $ a $ p $-linear map onto $ X $. Suppose there exists an $ \eta\in\pare{0, \frac{1}{4}} $ such that
\begin{equation}\label{eq:hyp_T1}
\norm{T_1}\leqslant \eta, 
\end{equation}
and a positive real number $ r $ such that
\begin{equation} \label{eq:hyp_r}
0< r < \min \set{\frac{1}{8 \norm{T_2}},\ \frac{1}{4\sqrt{\norm{T_3}}}}.
\end{equation}
 For any $ y\in B_X\pare{0, \frac{r}{4}} $, there exist a unique $ x\in B_X\pare{0,  r} $ such that 
\begin{equation*}
x= y + T_1\pare{x} + T_2\pare{x, x} + T_3\pare{x, x, x}. 
\end{equation*}
\end{prop}

\begin{rem}
Let us remark that we assume a smallness hypothesis (contractivity) on the linear operator $ T_1 $. Neglecting such hypothesis compromise irremediably the possibility of finding a fixed point via an iterative argument. \fine
\end{rem}

\begin{proof}
The proof of Proposition \ref{prop:fixed_point} is rather standard. Let us define inductively  the sequence
\begin{equation*}
\left\lbrace
\begin{aligned}
& x_0 =0, \\
& x_{n+1} = y + T_1 \pare{x_n} + T_2\pare{x_n, x_n} + T_3\pare{x_n, x_n, x_n}. 
\end{aligned}
\right.
\end{equation*}
We deduce immediately, thanks to \eqref{eq:hyp_T1} and \eqref{eq:hyp_r} that if $ x_n \in B_X\pare{0, r} $ then
\begin{equation*}
\norm{x_{n+1}}< r. 
\end{equation*}
Next we prove that the sequence $ \pare{x_n}_n $ is a Cauchy sequence in the topology of $ X $, since
\begin{multline*}
x_{n+1}-x_n = T_1\pare{x_n-x_{n-1}} + T_2 \pare{x_n, x_n-x_{n-1} } + T_2 \pare{ x_n-x_{n-1} , x_n }\\
 + T_3 \pare{x_n, x_n, x_n-x_{n-1} } + T_3 \pare{x_n, x_n-x_{n-1} ,  x_n } + T_3 \pare{ x_n-x_{n-1} , x_n ,  x_n }, 
\end{multline*}
we deduce, using the hypothesis \eqref{eq:hyp_T1} and \eqref{eq:hyp_r}
\begin{align*}
\norm{x_{n+1}-x_n} & \leqslant \pare{\Big. \eta + 2r \ \norm{T_2} + 3r^2\ \norm{T_3}}\norm{x_n-x_{n-1}}, \\
& < \frac{3}{4}\ \norm{x_n-x_{n-1}}, 
\end{align*}
which holds for any $ n\geqslant 1 $ and which indeed implies that $ \pare{x_n}_n $ is a Cauchy sequence in the Banach space $ X $, it is hence convergent. In order to prove uniqueness we suppose there exist two different $ x, z \in B_X \pare{0, r} $ so that
\begin{align*}
x & = y + T_1\pare{x} + T_2\pare{x, x} + T_3\pare{x, x, x}, \\
z & = y + T_1\pare{z} + T_2\pare{z, z} + T_3\pare{z, z, z}. 
\end{align*}
We subtract the two equations here above so that we obtain 
\begin{equation*}
x-z = T_1\pare{x-z} + T_2 \pare{x, x-z } + T_2 \pare{ x-z , z }
 + T_3 \pare{x, x, x-z } + T_3 \pare{x, x-z ,  z } + T_3 \pare{ x-z , z ,  z }.
\end{equation*}
Taking norms on the above equality using the triangular inequality and the fact that $ \norm{x}, \norm{z} < r $ we deduce
\begin{align*}
\norm{x-z} & \leqslant \pare{\Big. \eta + 2r \ \norm{T_2} + 3r^2\ \norm{T_3}}\norm{x-z}, \\
& < \frac{3}{4}\ \norm{x-z}, 
\end{align*}
which is obviously satisfied if and only if $ x=z $, concluding. 
\end{proof}

\subsection{Estimates for linear parabolic equations}\label{sec:linear_parabolic_estimates}

In the present section we prove some more or less well-known estimates for linear parabolic equations which will be of the utmost importance in the developement of the paper. \\

Let us consider two functions $ h, g $ defined on $ \bR $, and let us consider a $ T \in \left( 0, \infty \right] $.  We denote $ h\star g = 1_{\bra{0, T}} h \ast 1_{\bra{0, T}} g $ where $ \ast $ is the standard convolution. \\

In this section we will use continuously the \textit{Minkowsky integral inequality} : let us consider $ \pare{S_1, \mu_1} $ and $ \pare{S_2, \mu_2} $ two $ \sigma $--finite measure spaces and let $ f : S_1\times S_2 \to \bR $ be measurable, then the following inequality holds true:
\begin{equation*}
{\displaystyle \left[\int _{S_{2}}\left|\int _{S_{1}}f(x,y)\,\mu _{1}(\mathrm {d} x)\right|^{p}\mu _{2}(\mathrm {d} y)\right]^{\frac {1}{p}}\leqslant \int _{S_{1}}\left(\int _{S_{2}}|f(x,y)|^{p}\,\mu _{2}(\mathrm {d} y)\right)^{\frac {1}{p}}\mu _{1}(\mathrm {d} x).}
\end{equation*}

As an immediate application of the Minkowski integral inequality we can deduce the following result; 

 \begin{lemma} \label{lem:anis_Leb}
Let $1\leqslant p \leqslant p'$ and $f: X_1\times X_2 \to \mathbb{R}$ a function belonging to $L^p\left( X_1; L^{p'}\left( X_2\right) \right)$ where $\left( X_1; \mu_1\right),\left(X_2;\mu_2\right)$ are measurable spaces, then $f\in L^{p'}\left( X_2; L^p\left( X_1 \right)\right)$ and we have the inequality
$$
\left\|f\right\|_{L^{p'}\left( X_2; L^p\left( X_1 \right)\right)}\leqslant \left\|f\right\|_{L^p\left( X_1; L^{p'}\left( X_2\right) \right)}.
$$
\end{lemma}

Let us now consider the linear parabolic system with damping

\begin{equation}
\label{eq:parabolic_linear}
\left\lbrace
\begin{aligned}
& \partial_t w + \gamma w - \mu  \Delta w = F, \\
& \left. w\right|_{t=0} = w_0 . 
\end{aligned}
\right. 
\end{equation}

\noindent The estimates that we prove in this section are in particular focused to show \textit{quantitative} smoothing effects on the solutions of \eqref{eq:parabolic_linear} in terms of the parameters $ \gamma $ and $ \mu $. \\

The following result is classical, we refer to \cite[Lemma 5.10, p. 210]{BCD}:

\begin{lemma}\label{lem:parabolic1}
Let $ w $ be the unique solution of \eqref{eq:parabolic_linear} in $ \cC \pare{ \bra{0, T} ;  \cS' \pare{\bR^d}} $ of the Cauchy problem \eqref{eq:parabolic_linear} when $ \gamma \geqslant 0 $ with $ F \in L^2 \pare{ \bra{0, T} ;  \dot{H}^{s-1} \pare{\bR^d}} $ and let $ w_0\in \Hs \pare{\bR^d} $. Then for each $ t\in \bra{0, T} $
\begin{align*}
\norm{w\pare{t}}_{\Hs \pare{\bR^d}} ^2 + \mu \int_0^t \norm{\nabla w \pare{t'}}_{\Hs\pare{\bR^d}}^2 \d t' \leqslant \norm{w_0}_{\Hs\pare{\bR^d}}^2 + \frac{C}{\mu} \norm{F}_{L^2_T\dot{H}^{s-1}\pare{\bR^d}}^2, \\
\norm{w}_{L^4_T \dot{H}^{s+\frac{1}{2}}\pare{\bR^d}} \leqslant \frac{C}{\mu ^{ {1}/{4}}} \pare{ \norm{w_0}_{\Hs\pare{\bR^d}} + \frac{1}{\mu ^{ {1}/{2}}} \norm{F}_{L^2_T\dot{H}^{s-1}\pare{\bR^d}} } .
\end{align*}
\end{lemma}

For our purposes we will require the bulk force $ F $ appearing in \eqref{eq:parabolic_linear} to be in $ L^{ {4}/{3}}_T L^2 $, whence  Lemma \ref{lem:parabolic1} will not suffice in our context.  

\begin{lemma}\label{lem:parabolic2}
Let $ q\in\bra{1, 2}, \ T\in \left( 0, \infty\right] $ and let us define \begin{equation*}
s_q = 2\pare{ 1 -\frac{1}{q}}\ \in\bra{0, 1}, 
\end{equation*}
and let us suppose 
  $ F \in L^{ q} \pare{ \bra{0, T} ; \dot{H}^{s-s_q} \pare{\bR^d}} $ and let $ w_0\in \Hs \pare{\bR^d}\cap \dot{H}^{s+\frac{1}{2}} \pare{\bR^d} $.
Let us denote with $ w $ be the unique solution of \eqref{eq:parabolic_linear} in $ \cC \pare{ \bra{0, T} ;  \cS' \pare{\bR^d}} $ of the Cauchy problem \eqref{eq:parabolic_linear} when $ \gamma > 0 $.
 Then
\begin{equation*}
\norm{w}_{L^4_T\dot{H}^{s+\frac{1}{2}}\pare{\bR^d}} \leqslant C \bra{  \min \set{ \frac{\norm{w_0}_{\dot{H}^{s+\frac{1}{2}}\pare{\bR^d}}}{\gamma^{ {1}/{4}}}, \ \frac{\norm{w_0}_{\Hs}\pare{\bR^d}}{\mu^{1/4}}  }  + \frac{1}{\mu^{3/4}} \ \norm{F}_{L^{q}_T\dot{H}^{s-s_q}\pare{\bR^d}} }. 
\end{equation*}
\end{lemma}

\begin{proof}
Let us perform a $ \Hs\pare{\bR^d} $ estimate onto \eqref{eq:parabolic_linear}. We deduce the energy inequality

\begin{align*}
\frac{1}{2}\frac{\d}{\d t}\norm{w\pare{t}}_{\Hs}^2 + \gamma \norm{w\pare{t}}_{\Hs}^2  + \mu \norm{ \nabla w}_{\Hs}^2  
&\leqslant \av{\ps{F\pare{t}}{w\pare{t}}_{\Hs}}, \\
&\leqslant  \norm{F\pare{t}}_{\dot{H}^{s-s_q}}\norm{w\pare{t}}_{\dot{H}^{s+s_q}}. 
\end{align*}

\noindent Integrating the above relation in $ \bra{0, t}, \ t\in\bra{0, T} $ we deduce the inequality
\begin{equation}\label{eq:LPE1}
\frac{1}{2}\sup _{t'\in\pare{0, t}}\set{\norm{w\pare{t'}}_{\Hs}^2} + \gamma \int_0^t \norm{w\pare{t'}}_{\Hs}^2 \d t'  + \mu \int_0^t \norm{ \nabla w\pare{t'}}_{\Hs}^2 \d t' \leqslant \frac{1}{2}\norm{w_0}_{\Hs}^2 + \norm{F}_{L^{q}_t \dot{H}^{s-s_q}} \norm{w}_{L^{\frac{q}{q-1}}_t\dot{H}^{s+s_q}}. 
\end{equation}

\noindent A standard interpolation of Sobolev spaces implies that
\begin{equation*}
\norm{w}_{L^{\frac{q}{q-1}}_t\dot{H}^{s+s_q}} \leqslant \norm{w}_{L^\infty_t\Hs}^{\frac{2-q}{q}}\norm{\nabla w}_{L^2_t\Hs}^{\frac{2}{q}\pare{q-1}}, 
\end{equation*}
whence using the inequality
\begin{align*}
\norm{F}_{L^{q}_t \dot{H}^{s-s_q}} \norm{w}_{L^{\frac{q}{q-1}}_t\dot{H}^{s+s_q}} & 
\leqslant \norm{F}_{L^{q}_t \dot{H}^{s-s_q}}   \norm{w}_{L^\infty_t\Hs}^{\frac{2-q}{q}}\norm{\nabla w}_{L^2_t\Hs}^{\frac{2}{q}\pare{q-1}}, \\
&\leqslant \frac{1}{4} \norm{w}_{L^\infty_t\Hs}^{2} + \frac{3\mu}{4} \norm{w}_{L^2_t\Hs}^2 +\frac{C}{\mu} \norm{F}_{L^{q}_t \dot{H}^{s-s_q}}^2, 
\end{align*}
which inserted in \eqref{eq:LPE1} gives
\begin{equation*}
\frac{1}{4}\sup _{t'\in\pare{0, t}}\set{\norm{w\pare{t}}_{\Hs}^2} + {\gamma }\int_0^t \norm{w\pare{t'}}_{\Hs}^2 \d t'  + \frac{\mu}{4} \int_0^t \norm{ \nabla w\pare{t'}}_{\Hs}^2 \d t'\leqslant \frac{1}{2}\norm{w_0}_{\Hs}^2 + \frac{C}{\mu} \norm{F}_{L^{q}_t \dot{H}^{s-s_q}}^2. 
\end{equation*}

\noindent The above equation in particular implies that
\begin{equation}
\label{eq:interpolation1}
\begin{aligned}
\norm{w}_{L^2_t \dot{H}^{s+1}} & \leqslant \frac{C}{\mu ^{1/2}} \pare{
\norm{w_0}_{\Hs} + \frac{1}{\mu ^{1/2}} \norm{F}_{L^{q}_t \dot{H}^{s-s_q}}
}, \\
\norm{w}_{L^2_t \dot{H}^{s}} & \leqslant \frac{C}{\gamma ^{1/2}} \pare{
\norm{w_0}_{\Hs} + \frac{1}{\mu ^{1/2}} \norm{F}_{L^{q}_t \dot{H}^{s-s_q}}
}. 
\end{aligned}
\end{equation}


Let us now denote 
\begin{equation} \label{eq:propagator_Sgammamu}
S_{\gamma, \mu} \pare{ \partial, t} g \pare{x} = e^{-t\pare{\gamma - \mu \Delta}} g \pare{x}.
\end{equation}
Indeed the solution of equation \eqref{eq:parabolic_linear} can be expressed in terms of the evolution semigroup $ S_{\gamma, \mu} $ as 
\begin{equation}\label{eq:w_mild}
\hat{w}\pare{\xi, t} = S_{\gamma, \mu} \pare{ \xi, t} \hat{w}_0 \pare{\xi} + \int_0^t S_{\gamma, \mu} \pare{ \xi, t-t'} \hat{F}\pare{\xi , t'} \d t'. 
\end{equation}

\noindent An application of H\"older inequality give us the estimate, for $ t\in\bra{0, T} $
\begin{equation*}
\sup _{t'\in\pare{0, t}}\av{\hat{w}\pare{\xi, t'}} \leqslant \av{\hat{w}_0\pare{\xi}} + \frac{C_q}{\pare{\gamma + \mu\av{\xi}^2}^{\frac{q-1}{q}}}\ \norm{\hat{F}\pare{\xi, \cdot}}_{L^{q}\pare{\bra{0, t}}}, 
\end{equation*}
whence an $ L^2\pare{\bR^d, \av{\xi}^{2s}\d \xi} $ estimate on the above inequality allow us to deduce
\begin{equation}
\label{eq:LPE2}
\begin{aligned}
V\pare{t} & \overset{\text{def}}{=} \pare{\int _{\bR^d}\av{\xi}^{2s} \pare{\sup _{t'\in\pare{0, t}}\av{\hat{w}\pare{\xi, t'}}}\d \xi}^{1/2}, \\
& \leqslant \norm{w_0}_{\Hs} + \left( \int _{\bR^d} \frac{\av{\xi}^{2s}}{\pare{\gamma + \mu\av{\xi}^2}^{s_q}}\ \norm{\hat{F}\pare{\xi, \cdot}}_{L^{q}\pare{\bra{0, t}}}^2 \d \xi \right)^{1/2}
\end{aligned}
\end{equation}

\noindent Whence we remark that
\begin{align*}
\pare{
 \int _{\bR^d} \frac{\av{\xi}^{2s}}{2\pare{\gamma + \mu\av{\xi}^2}^{s_q}} \norm{\hat{F}\pare{\xi, \cdot}}^2_{L^{ q}_T} \d \xi
 }^{ {1}/{2}}
 &\leqslant \mu ^{-1/2}
 \pare{
 \int _{\bR^d} {\av{\xi}^{2\pare{s-s_q}}} \norm{\hat{F}\pare{\xi, \cdot}}^2_{L^{ q}_T} \d \xi
 }^{ {1}/{2}}\\
&  = \mu ^{-1/2} \norm{\hat{F}}_{L^2\pare{\bR^d, \  \av{\xi}^{\pare{s-s_q }} \d \xi ; \ L^{q}\pare{\bra{0, T}}}}. 
\end{align*}

\noindent We use hence Lemma \ref{lem:anis_Leb} with $ p=q $, $  \mu_2 \pare{\d \xi} = \av{\xi}^{\pare{s-s_q}} \d \xi $ and $ p'=2 $ to deduce that
\begin{equation}
\label{eq:LPE3}
\norm{\hat{F}}_{L^2\pare{\bR^d, \  \av{\xi}^{\pare{s-s_q}} \d \xi ; \ L^{q}\pare{\bra{0, T}}}} \leqslant \norm{F}_{L^{q}_T \dot{H}^{s-s_q}\pare{\bR^d} }, 
\end{equation}
and we use again Lemma \ref{lem:anis_Leb} in order to deduce 
\begin{equation}
\label{eq:LPE4}
\norm{w}_{L^\infty_t\Hs}\leqslant V\pare{t}.
\end{equation}

\noindent
Inserting the estimates \eqref{eq:LPE3} and \eqref{eq:LPE4} in \eqref{eq:LPE2} we deduce
\begin{equation}\label{eq:interpolation2}
\norm{w}_{L^\infty_t\Hs} \leqslant C_q \left( \norm{w_0}_{\Hs} + \frac{1}{\mu ^{1/2}} \norm{F}_{L^{q}_T \dot{H}^{s-s_q}\pare{\bR^d} } \right)
\end{equation}

Interpolating equation \eqref{eq:interpolation1} and \eqref{eq:interpolation2} we deduce that, for any $ t\in\bra{0, T} $
\begin{equation}\label{eq:interpolation3}
\begin{aligned}
\norm{w}_{L^p_t \dot{H}^{s+\frac{2}{p}}} & \leqslant \frac{C_q}{\mu ^{1/p}} \pare{
\norm{w_0}_{\Hs} + \frac{1}{\mu ^{1/2}} \norm{F}_{_{L^{q}_T \dot{H}^{s-s_q}\pare{\bR^d} }}
}, \\
\norm{w}_{L^p_t \dot{H}^{s}} & \leqslant \frac{C_q}{\gamma ^{1/p}} \pare{
\norm{w_0}_{\Hs} + \frac{1}{\mu ^{1/2}} \norm{F}_{_{L^{q}_T \dot{H}^{s-s_q}\pare{\bR^d} }}
}. 
\end{aligned}
\end{equation}

\noindent Setting $ p=4 $ in the first equation of \eqref{eq:interpolation3} we almost obtain the claim, what remains to be proved is the decaying effects on the initial data. Using Minkowski integral inequality and standard computations

\begin{equation}
\label{eq:bound_w0_L4Hs}
\begin{aligned}
\norm{S_{\gamma, \mu}\pare{\partial} w_0}_{L^4_T\dot{H}^{s+\frac{1}{2}}\pare{\bR^d}} & =  \pare{ \int_0^T \pare{\int_{\bR^d} \av{\xi}^{2s+1} e^{-2t\pare{\gamma + \mu \av{\xi}^2 }} \av{\hat{w}_0\pare{\xi}}^2 \d\xi}^2 \d t }^{ {1}/{4}} , \\
& \leqslant \pare{ \int_{\bR^d} \av{\xi}^{2s+1} \av{\hat{w}_0\pare{\xi}}^{2}\pare{ \int_0^T 
e^{-4t\pare{\gamma + \mu \av{\xi}^2 }}
}^{ {1}/{2}}  \d \xi }^{ {1}/{2}}, \\
& \leqslant \pare{ \int_{\bR^d} \frac{\av{\xi}^{2s+1}}{2\sqrt{\gamma + \mu\av{\xi}^2}} \av{\hat{w}_0\pare{\xi}}^{2} \d \xi }^{ {1}/{2}}, \\
& \leqslant C \ \min \set{ \frac{\norm{w_0}_{\dot{H}^{s+\frac{1}{2} \pare{\bR^d}}}}{\gamma^{ {1}/{4}}}, \frac{\norm{w_0}_{\dot{H}^{s}\pare{\bR^d}}}{\mu^{1/4}}  } . 
\end{aligned}
\end{equation}

%
%

\end{proof}

The next lemma describes the regularity of the solutions of  \eqref{eq:parabolic_linear} in the case in which the external force is in $ \LqHu $, whence we focus on the regularity induced by the damping $ \gamma w $ and we do not consider any space-smoothing effect induced by the heat propagator:

\begin{lemma}\label{lem:smoothin_bulk_force}
Let $ w_0\in \Hs  $ and let $ F\in L^2_T \Hs$, then $ w $ solution of \eqref{eq:parabolic_linear} is such that
\begin{equation*}
\norm{w}_{L^4_T \Hs\pare{\bR^d}}\leqslant \ C \pare{ \frac{1}{\gamma^{ {1}/{4}}} \ \norm{w_0}_{\Hs\pare{\bR^d}} + \frac{1}{\gamma^{3/4}} \norm{F}_{L^2_T \Hs\pare{\bR^d}} }
\end{equation*}
\end{lemma}

\begin{proof}
The present proof is a slight modification of the proof of Lemma \ref{lem:parabolic2}. \\

In the same way we deduced \eqref{eq:LPE2} we can argue that (here we set $ q=2 $ and $ t\in\bra{0, T} $) 
\begin{align*}
\norm{w}_{L^\infty_t \Hs}
 & \leqslant \norm{w_0}_{\Hs} + \left( \int _{\bR^d} \frac{\av{\xi}^{2s}}{{\gamma + \mu\av{\xi}^2}}\ \norm{\hat{F}\pare{\xi, \cdot}}_{L^{2}\pare{\bra{0, t}}}^2 \d \xi \right)^{1/2}, \\
 & \leqslant \norm{w_0}_{\Hs} +\frac{1}{\gamma^{1/2}}\norm{F}_{L^2_T \Hs}. 
\end{align*}

Performing a $ \Hs $ energy estimate on \eqref{eq:parabolic_linear} we deduce an estimate similar to \eqref{eq:interpolation1};
\begin{equation*}
\norm{w}_{L^2_t \dot{H}^{s}}  \leqslant \frac{C}{\gamma ^{1/2}} \pare{
\norm{w_0}_{\Hs} + \frac{1}{\gamma ^{1/2}} \norm{F}_{L^{2}_t \dot{H}^{s}}
}. 
\end{equation*}
An interpolation now concludes the estimates. 
\end{proof}

Lemma \ref{lem:smoothin_bulk_force} in particular asserts that, if $ F $ is sufficiently regular, the solution $ w $ of \eqref{eq:parabolic_linear} is an $ o_\gamma\pare{1} $ function in $ L^4_T\dot{H}^s $. This is not completely surprising, in fact supposing that $ F\in L^2_T\dot{H}^{s-1} $ (let us remark that such regularity is \textit{not} the same one required in the statement of Lemma \ref{lem:smoothin_bulk_force}) a standard $ \Hs $ energy estimate on the equation \eqref{eq:parabolic_linear} shows that in fact $ w $ is $ \cO\pare{\gamma^{-1}} $ as $ \gamma\to\infty $ in $ L^2_T\Hs $, interpolating hence we deduce that $ w $ is $ o_\gamma\pare{1} $ in $ L^p_T\Hs $ for $ p\in\left[2, \infty\right) $ (if $ F $ is "sufficiently regular"). This is obviously not the case when $ p=\infty $; the damping provided by the term $ \gamma w $ has no effect in $ t=0 $, we want though to quantify such damping effects for strictly positive times. \\

Let us now set $ \alpha,  \gamma, \mu > 0 $, and let us define the following function defined in $ \bR^d $
\begin{equation}\label{eq:definizione_m_gamma_mu}
m_{\gamma, \mu}^\alpha\pare{x} = \pare{\frac{\av{x}^2}{\gamma + \mu \av{x}^2}}^{\frac{\alpha}{2}}. 
\end{equation}

Indeed to the function $ m_{\gamma, \mu}^\alpha $ we can associate a Fourier multiplier 
\begin{equation*}
m_{\gamma, \mu}^\alpha \pare{\partial} g = 
\cF^{-1} \pare{ m_{\gamma, \mu}^\alpha \pare{\xi} \hat{g} \pare{\xi} }
= \cF^{-1} \pare{\pare{\frac{\av{\xi}^2}{\gamma + \mu \av{\xi}^2}}^{\frac{\alpha}{2}} \hat{g} \pare{\xi}}. 
\end{equation*}

\begin{lemma}\label{lem:properties_m_gamma}
Let $ g\in L^2\pare{\bR^d} $, then
\begin{equation*}
\norm{m_{\gamma, \mu}^\alpha\pare{\partial} g}_{L^2\pare{\bR^d}} \leqslant {\frac{C}{\gamma^{\alpha/4}} } \norm{g}_{L^2\pare{\bR^d}} +\frac{1}{{\mu^{\alpha/2}}} \ o_{\gamma}\pare{1}, 
\end{equation*}
where $ o_{\gamma}\pare{1} $ is a nonnegative function which tends to zero as $ \gamma $ tends to infinity. 
\end{lemma}

Lemma \ref{lem:properties_m_gamma} in particular asserts that, fixed $ \mu, \alpha >0 $,  $ m_{\gamma, \mu}^\alpha\pare{\partial}\xrightarrow{\gamma\to \infty} 0 $ as a linear operator on $ L^2\pare{\bR^d} $.

\begin{proof}
\begin{align*}
\norm{m_{\gamma, \mu}\pare{\partial} g}_{L^2\pare{\bR^d}}^2 & = 
\int \pare{\frac{\av{\xi}^2}{\gamma + \mu \av{\xi}^2}}^\alpha \av{\hat{g}\pare{\xi}}^2 \d\xi, \\
& = \int_{\av{\xi}\leqslant \gamma^{1/4}} \pare{\frac{\av{\xi}^2}{\gamma + \mu \av{\xi}^2}}^\alpha \av{\hat{g}\pare{\xi}}^2 \d\xi
+ \int_{\av{\xi}> \gamma^{1/4}} \pare{\frac{\av{\xi}^2}{\gamma + \mu \av{\xi}^2}}^\alpha \av{\hat{g}\pare{\xi}}^2 \d\xi = I_\gamma + I^\gamma. 
\end{align*}
Since $ g\in L^2 $ and since $ {m_{\gamma, \mu}^{2\alpha}}\leqslant \mu ^{-\alpha} $ pointwise  we can assert, by dominated convergence, that
\begin{equation*}
I^\gamma \leqslant \frac{1}{\mu^\alpha} \ o_\gamma\pare{1}, 
\end{equation*}
while since $ {m_{\gamma, \mu}^{2\alpha}} $ is strictly increasing in $ \av{\xi} $ we can assert that
\begin{equation*}
I_\gamma \leqslant \frac{1}{{\gamma}^{\alpha/2}} \ \norm{g}_{L^2\pare{\bR^d}}^2, 
\end{equation*}
concluding. 
\end{proof}

\begin{definition}
Given two Banach spaces $ X, Y $ we say that $ z\in X+Y $ if there exists a $ x\in X $ and an $ y\in Y $ so that $ z=x+y $. Moreover
\begin{equation*}
\norm{z}_{X+Y}= \sup \set{\norm{x}_X+\norm{y}_Y \  \left|\ x\in X, y\in Y\ \wedge \ x+y = z\Big. \right. } .
\end{equation*}
\end{definition}

Our aim is to use hence Lemma \ref{lem:properties_m_gamma} in order to study the damping properties, when $ \gamma $ is large, in $ L^\infty_T\Hs \pare{\bR^d} $ of the solutions of \eqref{eq:parabolic_linear} when $ F $ is an $ \cO\pare{1} $ function in some suitable space.

\begin{lemma}\label{lem:linear_damping_estimate}
Let $ w_0\in\Hud $ and $ F\in L^2_T \dot{H}^{-\frac{1}{2}} + L^{4/3}_TL^2 $, i.e. $ F=F_1+F_2 $ with $ F_1\in L^2_T \dot{H}^{-\frac{1}{2}}  $ and $ F_2 \in L^{4/3}_T L^2 $. Let $ w $ be the unique tempered distribution which solves \eqref{eq:parabolic_linear}, then for each $ t\in\bra{0, T} $
\begin{equation}\label{eq:linear_damping_estimate}
\norm{w \pare{t}}_{\Hud} \leqslant C\pare{ e^{-\gamma t}\norm{w_0}_{\Hud} +{\frac{1}{\min\set{\gamma^{1/4}, \gamma^{1/8}}} } \norm{F}_{L^2_T \dot{H}^{-\frac{1}{2}}+ L^{4/3}_TL^2} +\frac{1}{ \min\set{\mu ^{1/2}, \mu^{1/4}} } \ o_{\gamma}\pare{1} }, 
\end{equation}
whence
\begin{equation}\label{eq:linear_damping_limit}
\lim _{\gamma\to \infty} \norm{w}_{L^\infty\pare{\pare{\varepsilon, T}; \Hud}} = 0,
\end{equation}
for any $ \varepsilon >0 $. 
\end{lemma}

\begin{rem}
Indeed the limit in \eqref{eq:linear_damping_limit} holds in the timespan $ \pare{\varepsilon, T} $ as it is clear from the estimate \eqref{eq:linear_damping_estimate}: in $ t=0 $ there is obviously no damping effect. \fine
\end{rem}

\begin{proof}

By superposition we can write $ w =W + w_1 + w_2 $, where
\begin{align*}
 W\pare{x, t} & = S_{\gamma, \mu} \pare{ \partial, t} w_0 \pare{x}, \\
 w_1\pare{x, t} & = \int_0^t S_{\gamma, \mu} \pare{ \partial, t-t'} F_1\pare{x, t'} \d t', \\
 w_2\pare{x, t} & = \int_0^t S_{\gamma, \mu} \pare{ \partial, t-t'} F_2\pare{x, t'} \d t'.
\end{align*}

Indeed the following bound is immediate
\begin{equation*}
\norm{W\pare{t}}_{\Hud} = \norm{S_{\gamma, \mu} \pare{ \partial, t} w_0}_{\Hud} \leqslant e^{-\gamma t}\norm{w_0}_{\Hud}. 
\end{equation*}

For $ w_1 $ we can argue as in \eqref{eq:interpolation1} (here we set $ q=2 $) in order to deduce
\begin{align*}
\norm{w_1}_{L^\infty_t \Hud} & \leqslant \left( \int _{\bR^d} \frac{\av{\xi}}{{\gamma + \mu\av{\xi}^2}}\ \norm{\hat{F}_1\pare{\xi, \cdot}}_{L^{2}\pare{\bra{0, t}}}^2 \d \xi \right)^{1/2}, \\
& = \norm{m_{\gamma, \mu}^1\pare{\partial}\partial^{-1/2}\norm{F_1}_{L^2_t}}_{L^2_x}. 
\end{align*}
We apply Lemma \ref{lem:properties_m_gamma} with $ \alpha = 1 $ in order to deduce
\begin{align*}
\norm{w_1}_{L^\infty_t \Hud}  & \leqslant \frac{C}{\gamma^{1/4}} \norm{F_1}_{L^2_t\Hmud} + \frac{1}{\mu ^{1/2}} o_\gamma\pare{1}. 
\end{align*}
\noindent
For $ w_2 $ we repeat the same procedure which lead us to prove \eqref{eq:interpolation1}, setting $ q=4/3 $, we have
\begin{align*}
\norm{w_2}_{L^\infty_t \Hud} & \leqslant \left( \int _{\bR^d} \frac{\av{\xi}}{\pare{ \gamma + \mu\av{\xi}^2}^{1/2}}\ \norm{\hat{F}_2\pare{\xi, \cdot}}_{L^{4/3}\pare{\bra{0, t}}}^2 \d \xi \right)^{1/2}, \\
& = \norm{m_{\gamma, \mu}^{1/2}\pare{\xi}\norm{\hat{F}_2}_{L^{4/3}_t}}_{L^2_\xi}. 
\end{align*}
We again use Lemma \ref{lem:properties_m_gamma} with $ \alpha=1/2 $, next Lemma \ref{lem:anis_Leb} and Plancherel theorem to deduce the final bound required
\begin{align*}
\norm{w_2}_{L^\infty_t \Hud}  & \leqslant \frac{C}{\gamma^{1/8}} \norm{F_2}_{L^{4/3}_t L^2} + \frac{1}{\mu ^{1/4}} o_\gamma\pare{1}. 
\end{align*}

\end{proof}

\section{Reformulation of the system \eqref{eq:Shilomis1}}\label{sec:ref_syst}

As already mentioned the main goal in the present study is to study the dynamics of the system \eqref{eq:Shilomis1} when $ \tau $ is small or tends to zero. Intuitively one understands that, when $ \tau\to 0 $ the term 
\begin{equation*}
\frac{1}{\tau}\pare{M-\chi_0 H}, 
\end{equation*}
is the leading order term (in $ \tau $) in \eqref{eq:Shilomis1}, whence we expect, when $ \tau $ is sufficiently close to zero, to have the asymptotic development $ M-\chi_0 H = \cO\pare{\tau} $ in some suitable topology. To understand rigorously  this asymptotic is the mayor difficulty in the analysis of solutions of \eqref{eq:Shilomis1}. \\

Heuristically one expects the term  $ \frac{1}{\tau}\pare{M-\chi_0 H} $ to provide a damping effect on the components $ M, H $, solutions of \eqref{eq:Shilomis1}. The damping effect is though not immediately clear from \eqref{eq:Shilomis1}; the aim of the present section is hence to provide a new reformulation of the system \eqref{eq:Shilomis1} in some new, but equivalent, unknowns which explicit the damping effect provided by the term $ \frac{1}{\tau}\pare{M-\chi_0 H} $. \\

From the magnetostatic equation, the third equation of \ref{eq:Shilomis1}, and since $ \curl \ H=0 $, we immediately deduce that
\begin{equation*}
H = -\cQ M  + \cG_F, \hspace{1cm} \cG_F = \nabla\Delta^{-1}F,
\end{equation*}
where $ \cQ=\Delta^{-1}\nabla\div $.

\begin{rem}
Let us remark that if $ \cG_F = \nabla\Delta^{-1} F $ and $ F $ has the regularity stated in \eqref{eq:regularity_F} then
\begin{align}\label{eq:regularity_GF}
\cG_F \in L^4_T  \dot{H}^1 \cap L^2_T\dot{H}^3, && \partial_t \cG_F \in L^2_T\Hu. 
\end{align}
The regularity stated in \eqref{eq:regularity_GF} will be considered implicitly given from now on.  \fine
\end{rem}

Whence it is clear that, denoting $ \cP $ the Leray projector onto divergence-free vector fields, and denoting
\begin{align*}
m = \cP M, && \tilde{m} = \cQ M, 
\end{align*}
that
\begin{align*}
\frac{1}{\tau}\cP \pare{M-\chi_0 H} = \frac{1}{\tau} \ m , &&
\frac{1}{\tau}\cQ \pare{M-\chi_0 H} = \frac{1+\chi_0}{\tau} \bra{\tilde{m}-\frac{\chi_0}{1+\chi_0} \cG_F}. 
\end{align*}

We can hence define the new unknown 
\begin{equation*}
r = \tilde{m}-\frac{\chi_0}{1+\chi_0} \cG_F, 
\end{equation*}
of which we can compute the evolution equation from \eqref{eq:Shilomis1}. The advantage of working with the variables $  m , r $ instead than $ M, H $ resided in the fact that, for such, the damping induced by the term $ \frac{1}{\tau} \pare{M-\chi_0 H} $ is explicit. \\

We can hence compute the evolution equations for $ \pare{u, m, r} $ form the ones of $ \pare{u, M, H} $ (and vice-versa) via the following reversible change of variables
\begin{align}\label{eq:change_unknown}
\left\lbrace
\begin{aligned}
& u=u \\
& M = m+r + \frac{\chi_0}{1+\chi_0} \cG_F \\
& H = -r + \frac{1}{1+\chi_0} \cG_F
\end{aligned}
\right. ,
&&
\left\lbrace
\begin{aligned}
&u=u \\
& m = \cP M \\
& r = \cQ M - \frac{\chi_0}{1-\chi_0}\cG_F 
\end{aligned}
\right. .
\end{align}

Thanks to the explicit change of unknown given in \ref{eq:change_unknown} it is rather simple to deduce the evolution of $ \pare{u, m, r} $ from \eqref{eq:Shilomis1}, and we obtain
	\begin{equation}\tag{S2}\label{eq:Shilomis2}
		\left\lbrace
		\begin{aligned}
			& \begin{multlined}
			{\partial_t u + \pare{u\cdot\nabla} u}- \nu \Delta u + \nabla p = \pare{m+r + \frac{\chi_0}{1+\chi_0} \cG_F}\cdot \nabla \pare{-r + \frac{1}{1+\chi_0} \cG_F}\\
 			+ \frac{1}{2}\curl \bra{\pare{m+r + \frac{\chi_0}{1+\chi_0} \cG_F} \times \pare{-r + \frac{1}{1+\chi_0} \cG_F}},
\end{multlined} \\[5mm]
			& \begin{multlined}
			\partial_t m +\frac{1}{\tau} \ m - \sigma \Delta m = -\cP \bra{u\cdot\nabla \pare{ m+r + \frac{\chi_0}{1+\chi_0} \cG_F }} + \frac{1}{2}\cP \bra{ \pare{\curl \ u} \times \pare{m+r + \frac{\chi_0}{1+\chi_0} \cG_F}  } \\
			- \cP \set{ \pare{m+r + \frac{\chi_0}{1+\chi_0} \cG_F} \times \bra{ \pare{m+r + \frac{\chi_0}{1+\chi_0} \cG_F} \times \pare{-r + \frac{1}{1+\chi_0} \cG_F}} },
			\end{multlined}\\[5mm]
			& \begin{multlined}
			\partial_t r +\frac{1+\chi_0}{\tau} \ r - \sigma \Delta r = -\cQ \bra{u\cdot\nabla \pare{ m+r + \frac{\chi_0}{1+\chi_0} \cG_F }} + \frac{1}{2}\cQ \bra{ \pare{\curl \ u} \times \pare{m+r + \frac{\chi_0}{1+\chi_0} \cG_F}  } \\
			- \cQ \set{ \pare{m+r + \frac{\chi_0}{1+\chi_0} \cG_F} \times \bra{ \pare{m+r + \frac{\chi_0}{1+\chi_0} \cG_F} \times \pare{-r + \frac{1}{1+\chi_0} \cG_F}} }\\
			-\frac{\chi_0}{1+\chi_0}\pare{ \partial_t \cG_F - \sigma \Delta \cG_F \Big.},
			\end{multlined}\\
			& \div \ u=0, \\
			& \left.\pare{u, m, r}\right|_{t=0} = \pare{u_0, m_0, r_0}.
		\end{aligned}
		\right. 
	\end{equation}
	
	From now on we will work with the system in the form \eqref{eq:Shilomis2}.

	\begin{rem}\label{rem:explanation_nonlinear_terms}
	We would like to remark the fact that, despite the system \ref{eq:Shilomis2} seems at a firs sight much more complex than the system \eqref{eq:Shilomis1}, there is in fact no relevant new technical difficulty in \eqref{eq:Shilomis2}. \\
	In fact the nonlinearities appearing on the right hand side of \eqref{eq:Shilomis2} belong at most to six classes which we can study without problem and which are here enumerated\footnote{Here and in the rest of the paper we use Einstein summation convention}
	\begin{itemize}
	
	\item They can be of the form
		\begin{equation*}
		B_{\ns}\pare{v, v} =  \pare{ v_i \ q_{i, j}^{\ns, \ell}\pare{\partial} v_j}_{\ell=1, 2, 3}, 
		\end{equation*}
	where $ q_{i, j}^{\ns, \ell} $ are homogeneous Fourier multipliers of order one. 
	
	\item They can be of the form
	\begin{align*}
	\cL_1 \pare{v} = \pare{v_i \ q_{i, j}^{\pare{1}, \ell}\pare{\partial} G_j}_{\ell=1, 2, 3}, &&
	\cL_2 \pare{v} =  \pare{ G_i \ q_{i, j}^{\pare{2}, \ell}\pare{\partial}v_j}_{\ell=1, 2, 3},
	\end{align*}
	form some function $ G $ (notably in \eqref{eq:Shilomis2} $ G=\cG_F $). Here again $ q_{i, j}^{\pare{k}, \ell}, \ k =1, 2 $ are homogeneous Fourier multipliers of order one. 
	
	\item Lastly they can be $ p $-linear forms  of the form
	\begin{align*}
	\cN_p \pare{v} = v^{\otimes p} \otimes G^{\otimes\pare{3-p}},
	\end{align*}
	where we recall that given a $ w\in\bR^3 $ we identify as $ w^{\otimes q} $ the canonical $ q $--linear form whose components are elements of the form
	\begin{equation*}
	\prod_{q'=1}^q w_{j_{q'}}. 
	\end{equation*}
	In particular hence the components of $ \cN_p\pare{v} $ are of the form
	\begin{equation*}
	\prod_{q'=1}^p\prod_{q''=1}^{3-p} v_{j_{q'}} G_{j_{q''}}. 
	\end{equation*}

	\end{itemize}	
	
	Whence we can assert that \eqref{eq:Shilomis2} can be studied as a special system of the form
	\begin{equation}\label{eq:Shilomis_generic}
	\partial_t v + \bM v - \cA_2 \pare{\partial} v = B_{\ns}\pare{v, v} + \cL_1 \pare{v} + \cL_2 \pare{v} + \sum_{p=1}^3 \cN_p \pare{v} + f, 
	\end{equation}
	where $ \bM $ is a diagonal, nonnegative matrix and $ \cA_2 \pare{\partial} $ is an elliptic differential homogeneous operator of order two and $ f $ is a bulk force. We will many times think of \eqref{eq:Shilomis2} in the form \eqref{eq:Shilomis_generic} since there are much less terms to consider, which qualitatively describe every nonlinear term appearing in \eqref{eq:Shilomis2}. \fine
	\end{rem}

\section{Existence of a unique solution in a critical functional space uniformly in $ \tau\in\pare{0, \tau_0} $ } \label{sec:existence_sol}

In the present section we prove the main result of the paper, i.e. Theorem \ref{thm:main_thm}. The detailed result proved is the following one, which implies the proof of Theorem \ref{thm:main_thm} as explained in Remark \ref{rem:prop→thm}; 

\begin{prop}\label{prop:existence_unique_solution}
Let $ u_0\in\Hud $ and $ \cG_F\in L^4_{\loc}\pare{\bR_+; \Hu}\cap W^{1, 2}_{\loc}\pare{\bR_+; \Hu \cap \dot{H}^3} $. There exists a $ \rho, \varrho_0 >0 $ such that $ \rho > 2\varrho_0 $ and a $ T=T_{\varrho_0} \in\left( 0, \infty \right] $ (see \eqref{eq:def_T}) so that
\begin{equation} \label{eq:smallness_cG}
\norm{\cG}_{\LqHu}\leqslant \varrho_0 \leqslant \frac{\min\set{c^{1/2},\big.  c^{3/4}}}{C}, 
\end{equation}
where $ c = \min\set{\nu, \sigma} $ such that, if;

\begin{enumerate}[{a)}]

\item Let $ u_0, m_0, r_0 \in \Hud $ be such that
\begin{align*}
\norm{u_0}_{\Hud}\leqslant \frac{\nu^{1/4}}{C} \ \rho,
&&
\norm{\pare{m_0, r_0}}_{\Hud}\leqslant\frac{\sigma^{1/4}}{C}\ \rho,  
\end{align*}
and
\begin{equation}\label{eq:smallness_tau1}
\tau < \frac{\pare{1+\chi_0}^{7/3}}{ C\ \chi_0^{4/3}} \pare{\norm{\cG_F}_{L^2_T\dot{H}^{3}}   + \norm{\cG_F}_{\dot{W}^{1, 2}_T \dot{H}^1}}^{-4/3} \ \varrho_0^{4/3}, 
\end{equation}
then there exist a unique solution $ \pare{u, m, r} $ of \eqref{eq:Shilomis2} in the ball $ B\pare{0, 4\rho} $ of the space $ \LqHu $ which moreover belongs to the space $ \cC_T\Hud $. 

\item Let $ u_0, m_0, r_0 \in \Hud $ arbitrarily large and $ \tau > 0 $ satisfy the relation \eqref{eq:smallness_tau1}, there exist a $ T^\star = T^\star _{U_0}\in\pare{0, T} $ such that the system \eqref{eq:Shilomis2} admits a unique solution in the ball $ B\pare{0, 4\rho} $ of the space $ L^4_{T^\star}\dot{H}^1 $ which moreover belongs to the space $ \cC_{T^\star}\Hud $. 

\item \label{en:point_we_prove_in_the_ecistence_thm}  Let  $ u_0\in\Hud $ 
\begin{equation}\label{eq:smallness_vel_flow}
\norm{u_0}_{ \Hud}\leqslant\frac{\nu^{1/4}}{C} \ \rho, 
\end{equation}
and $ m_0, r_0\in\dot{H}^1 $  arbitrary. Let $ \tau $ be sufficiently small so that
\begin{equation}\label{eq:smallness_tau}
\tau\leqslant \min\set{
\frac{\rho^4}{C\pare{\norm{m_0}_{\dot{H}^1}^4 + \norm{r_0}_{\dot{H}^1}^4}}
\hspace{2mm}, \hspace{2mm}
 \frac{ \pare{{1+\chi_0}}^{7/3}\varrho_0^{4/3}}{C\chi_0^{4/3} \  \pare{\norm{\cG_F}_{L^2_T\dot{H}^{3}}   + \norm{\cG_F}_{\dot{W}^{1, 2}_T \dot{H}^1}}^{4/3}}
}. 
\end{equation}
Then there exist a unique solution $ \pare{u, m, r} $ of \eqref{eq:Shilomis2} in the ball $ B\pare{0, 4\rho} $ of the space $ \LqHu $ which moreover belongs to the space $ \cC_{T}\Hud $.

\item Let $ u_0\in\Hud $ arbitrarily large and let $ \tau $ satisfy \eqref{eq:smallness_tau}, there exist a $ T^\star\in\pare{0, T} $ such that the system \eqref{eq:Shilomis2} admits a unique solution in the ball $ B\pare{0, 4\rho} $ of the space $ L^4_{T^\star}\dot{H}^1 $ which moreover belongs to the space $ \cC_{T^\star}\Hud $.

\end{enumerate}

\end{prop}

\begin{rem}\label{rem:prop→thm}
Let us note that if we prove Proposition \ref{prop:existence_unique_solution} than we prove as well Theorem \ref{thm:main_thm} with the substitution
\begin{equation*}
\pare{u, m , r}\mapsto \pare{u, M , H}, 
\end{equation*}
defined explicitly in \eqref{eq:change_unknown}. \fine
\end{rem}

\begin{rem}
We will prove only the point \ref{en:point_we_prove_in_the_ecistence_thm} since the other points are variations of the same argument which are simple to the reader familiar with the construction of solutions for the \NS equations via a fixed point theorem.~\fine
\end{rem}

\begin{rem}
Let us point out that if we allow $ T=\infty $ in the statement of Proposition \ref{prop:existence_unique_solution} (i.e. it suffice to consider $ \cG_F $ to be "small" in $ \LqHu $) the points a and \ref{en:point_we_prove_in_the_ecistence_thm} provide a \textit{global} solution of \eqref{eq:Shilomis2}, in particular the point \ref{en:point_we_prove_in_the_ecistence_thm} provides a global  solution imposing a smallness hypothesis on $ u_0 $ \textit{only} in $ \Hud $ and assuming $ m_0, r_0 $ be \textit{arbitrarily large or unbounded} in $ \Hud $.~\fine
\end{rem}

The proof of the point \ref{en:point_we_prove_in_the_ecistence_thm} of Proposition \ref{prop:existence_unique_solution} is an application of the fixed point theorem stated in Proposition \ref{prop:fixed_point}; conceptually there is no great difference with the more familiar construction of a unique solution in a critical space for the incompressible \NS equations, there are though two main difficulties which we want to consider
\begin{itemize}

\item Indeed the nonlinear estimates for \eqref{eq:Shilomis2} are  more lengthy and complicated than the transport bilinear form of the incompressible \NS equations, 

\item Secondly, and more important in our context, we want to give a proof which provides an existence result which is uniform-in-time with respect to the physical parameter $ \tau \in\pare{0, \tau_0 } $ for some small $ \tau_0>0 $.

\end{itemize}

The proof is hence divided as follows:

\begin{itemize}

\item In Section \ref{sec:mild_form_Shilomis} we reformulate the system \eqref{eq:Shilomis2} in a suitable mild form. Such passage consist mostly in  computations which have to be carried out in detail due to the many nonlinearities appearing in system \eqref{eq:Shilomis2}, 

\item In Section \ref{sec:par_est_for_S} we provide some nonlinear parabolic estimates for the six generic classes of nonlinearities which compose all the nonlinear terms of Shilomis system \eqref{eq:Shilomis2}, as explained in Remark \ref{rem:explanation_nonlinear_terms}. Indeed the linear parabolic estimates carried out in the introductory Section \ref{sec:linear_parabolic_estimates} will be the main tool in order to prove the nonlinear estimates required, 

\item In Section \ref{sec:nonlinear_bounds_Shilomis} we apply the nonlinear parabolic bounds deduced in Section \ref{sec:par_est_for_S} to the mild form of \eqref{eq:Shilomis2} deduced in Section \ref{sec:mild_form_Shilomis}, 

\item Finally in Section \ref{sec:fixed_poin_application} we apply the nonlinear bounds for the Shilomis system deduced in Section \ref{sec:nonlinear_bounds_Shilomis} in order to apply the fixed point theorem stated in Proposition \ref{prop:fixed_point} and to deduce the existence of a unique solution of \eqref{eq:Shilomis2} in critical space.

\end{itemize}

\begin{rem}
Since the proof of Proposition \ref{prop:existence_unique_solution} relies on a fixed point theorem it is known that such result generally relies on a smallness hypothesis on which it is possible to construct a perturbative argument. \\
The smallness hypothesis appearing in Proposition \ref{prop:existence_unique_solution} are rather unusual, hence we would like to comment them:
\begin{enumerate}[$ \triangleright $]

\item The smallness hypothesis on the initial velocity flow \eqref{eq:smallness_vel_flow} is rather standard in the theory of \NS equations. 

\item The smallness hypothesis \eqref{eq:smallness_cG} can look peculiar in a first stance, but it is inevitable. It says in fact that the external magnetic field cannot pump too much $ \LqHu $ energy in the system. This is reasonable since in the equation \eqref{eq:Shilomis2} there are terms of the form $ \cG_F \cdot\nabla\cG_F $, if such term is arbitrarily large it will break down any smallness condition on which the perturbative argument for \NS equations is based; relaxing \eqref{eq:smallness_cG} is hence impossible in our context. 

\item As a matter of facts the in the point \ref{en:point_we_prove_in_the_ecistence_thm} we consider initial data $ m_0, r_0 $ arbitrarily large in $ \Hud $ and $ \dot{H}^1 $. Such hypothesis may look as unreasonable at a first sight, but we want to make notice to the reader that the smallness hypothesis \eqref{eq:smallness_tau} compensates to such lack of smallness for the initial data. In a nutshell it says that if the damping coefficient is sufficiently large the $ \dot{H}^1 $ norm of $ \pare{m, r} $ is damped with sufficient vigor so that $ \pare{m, r} $ turns out to be "small" in the space $ \LqHu $, without hence violating the smallness principle on which any perturbative method is based. 

\end{enumerate}
\end{rem}

\subsection{Reformulation of \eqref{eq:Shilomis2} in an appropriate mild form }\label{sec:mild_form_Shilomis}

Lt us rewrite the system \eqref{eq:Shilomis2} in the mild form

\begin{equation}
\label{eq:Shilomis_mild}
\left\lbrace
\begin{aligned}
 & u \pare{x, t} = S_{0, \nu}\pare{\partial, t} u_0\pare{x} + \int_0^t S_{0, \nu}\pare{\partial, t -t'} \cN_u\pare{x, t'} \d t', \\
 & m \pare{x, t} = S_{\frac{1}{\tau}, \sigma}\pare{\partial, t} m_0\pare{x} + \int_0^t S_{\frac{1}{\tau}, \sigma}\pare{\partial, t -t'} \cN_m\pare{x, t'} \d t', \\
 & r \pare{x, t} = S_{\frac{1+\chi_0}{\tau}, \sigma}\pare{\partial, t} r_0\pare{x} + \int_0^t S_{\frac{1+\chi_0}{\tau}, \sigma}\pare{\partial, t -t'} \cN_r\pare{x, t'} \d t'
 + \int_0^t S_{\frac{1+\chi_0}{\tau}, \sigma}\pare{\partial, t -t'} f\pare{x, t'} \d t', 
\end{aligned}
\right.
\end{equation}
where
\begin{equation}
\label{eq:nonlinearities_Shilomis_mild}
\begin{aligned}
& \begin{multlined}
\cN_u = - \cP \pare{u\cdot\nabla u } + \cP\bra{ \pare{m+r + \frac{\chi_0}{1+\chi_0} \cG_F}\cdot \nabla \pare{-r + \frac{1}{1+\chi_0} \cG_F}}\\
 			+ \frac{1}{2}\cP \ \curl \bra{\pare{m+r + \frac{\chi_0}{1+\chi_0} \cG_F} \times \pare{-r + \frac{1}{1+\chi_0} \cG_F}},
\end{multlined}\\[5mm]
& \begin{multlined}
\cN_m = -\cP \bra{u\cdot\nabla \pare{ m+r + \frac{\chi_0}{1+\chi_0} \cG_F }} + \frac{1}{2}\cP \bra{ \pare{\curl \ u} \times \pare{m+r + \frac{\chi_0}{1+\chi_0} \cG_F}  } \\
			- \cP \set{ \pare{m+r + \frac{\chi_0}{1+\chi_0} \cG_F} \times \bra{ \pare{m+r + \frac{\chi_0}{1+\chi_0} \cG_F} \times \pare{-r + \frac{1}{1+\chi_0} \cG_F}} },
\end{multlined} \\[5mm]
& \begin{multlined}
\cN_r = -\cQ \bra{u\cdot\nabla \pare{ m+r + \frac{\chi_0}{1+\chi_0} \cG_F }} + \frac{1}{2}\cQ \bra{ \pare{\curl \ u} \times \pare{m+r + \frac{\chi_0}{1+\chi_0} \cG_F}  } \\
			\hspace{8mm}  - \cQ \set{ \pare{m+r + \frac{\chi_0}{1+\chi_0} \cG_F} \times \bra{ \pare{m+r + \frac{\chi_0}{1+\chi_0} \cG_F} \times \pare{-r + \frac{1}{1+\chi_0} \cG_F}} }
			,\end{multlined}
\\[5mm]
&f = -\frac{\chi_0}{1+\chi_0}\pare{ \partial_t \cG_F - \sigma \Delta \cG_F \Big.}.
\\
\end{aligned}
\end{equation}

We will now reformulate the integral system \eqref{eq:Shilomis_mild} in an even more generic form with which will be easier to study. 
Let us now denote as $ U = \pare{u, m, r} $, the system \eqref{eq:Shilomis_mild} can alternatively be written as
\begin{equation}\label{eq:Shilomis_mild2}
U\pare{x, t} = \cS \pare{\partial, t}U_0\pare{x} + \mathcal{T}\bra{U}\pare{x, t} + g \pare{x, t}, 
\end{equation}
where
\begin{equation}\label{eq:cSU0}
\cS \pare{\partial, t}U_0\pare{x} = \pare{
\begin{array}{c}
S_{0, \nu}\pare{\partial, t} u_0\pare{x} \\
 S_{\frac{1}{\tau}, \sigma}\pare{\partial, t} m_0\pare{x} \\
 S_{\frac{1+\chi_0}{\tau}, \sigma}\pare{\partial, t} r_0\pare{x}
\end{array}
}, 
\end{equation}
while
\begin{equation}\label{eq:def_cT}
\cT\bra{U} = \left( \sum_{p=1}^3\cT_{p}\bra{U} \right) + \cT_{2, \ns} \bra{U} + \cT_{1, \RN{1}} \bra{U} + \cT_{1, \RN{2}} \bra{U}, 
\end{equation}
where
\begin{equation}\label{eq:cT2NS}
\cT_{2, \ns} \bra{U} = \pare{
\begin{array}{c}
\displaystyle \int_0^t S_{0, \nu}\pare{\partial, t -t'} \set{- \cP \pare{\big.u\cdot\nabla u } - \cP\bra{ \big. \pare{m+r }\cdot \nabla {r }}
 			- \frac{1}{2}\cP \ \curl \bra{\big.\pare{m+r } \times {r }}}\pare{t'} \d t' \\[5mm]
 			\displaystyle \int_0^t S_{\frac{1}{\tau}, \sigma}\pare{\partial, t -t'} \set{-\cP \bra{  \big. u\cdot\nabla \pare{ m+r }} + \frac{1}{2}\cP \bra{ \big. \pare{\curl \ u} \times \pare{m+r }  } }\pare{t'} \d t' \\[5mm]
 		\displaystyle \int_0^t S_{\frac{1+\chi_0}{\tau}, \sigma}\pare{\partial, t -t'} \set{ -\cQ \bra{u\cdot\nabla \pare{ m+r }} + \frac{1}{2}\cQ \bra{ \pare{\curl \ u} \times \pare{m+r    }}}\pare{t'}\d t'
\end{array}
},
\end{equation}
\begin{equation}\label{eq:cT1I}
\cT_{1, \RN{1}}\bra{U} = \pare{
\begin{array}{c}
\displaystyle \frac{1}{2 \pare{ 1+\chi_0}} \int_0^t S_{0, \nu}\pare{\partial, t -t'} \set{\Big.  
\cP\bra{ \pare{m+\pare{ 1-\chi_0}r }\cdot \nabla {  \cG_F}}
 			+ \cP  \bra{\pare{m+\pare{1+\chi_0}r}\div\ \cG_F}}\pare{t'} \d t' \\[5mm]
 			\displaystyle-\frac{\chi_0}{1+\chi_0}\int_0^t S_{\frac{1}{\tau}, \sigma}\pare{\partial, t -t'} \set{\cP\bra{\big. u\cdot\nabla \cG_F}}\pare{t'} \d t' \\[5mm]
 			\displaystyle-\frac{\chi_0}{1+\chi_0}\int_0^t S_{\frac{1+\chi_0}{\tau}, \sigma}\pare{\partial, t -t'} \set{\cQ \bra{\big. u\cdot\nabla \cG_F}}\pare{t'} \d t' 
\end{array}
},
\end{equation}
and let us remark how the operator  $ \cT_{1, \RN{1}}\bra{U} $ acts as a derivative  on the function $ \cG_F $ only, while the operator $  \cT_{1, \RN{2}}\bra{U} $ is defined as
\begin{equation}\label{eq:cT1II}
 \cT_{1, \RN{2}}\bra{U} = \pare{
\begin{array}{c}
\displaystyle \frac{1}{2 \pare{ 1+\chi_0}} \int_0^t S_{0, \nu}\pare{\partial, t -t'} \set{\Big. -\cG_{F} \div\pare{m+\pare{1+\chi_0}r} -\cG_{F}\cdot \nabla\pare{m+\pare{1+\chi_0}r}}\pare{t'} \d t'\\[5mm]
 			\displaystyle \frac{\chi_0}{2\pare{ 1+\chi_0}}\int_0^t S_{\frac{1}{\tau}, \sigma}\pare{\partial, t -t'} \set{\Big. \cP\bra{\curl u \times \cG_F}}\pare{t'} \d t' \\[5mm]
 			\displaystyle\frac{\chi_0}{2\pare{ 1+\chi_0}}\int_0^t S_{\frac{1+\chi_0}{\tau}, \sigma}\pare{\partial, t -t'} \set{\Big. \cQ\bra{\curl u \times \cG_F}}\pare{t'} \d t' 
\end{array} 
 }. 
\end{equation}

\noindent We now define the $ p $--linear operators $ \cT_p\bra{U} $; 
\begin{equation}\label{eq:cT1}
\cT_1\bra{U} = \pare{
\begin{array}{c}
0 \\[5mm]
\displaystyle \frac{\chi_0}{\pare{1+\chi_0}^2}\int_0^t S_{\frac{1}{\tau}, \sigma}\pare{\partial, t -t'} \set{ \cP\bra{\Big.\cG_F \times\pare{\big. \pare{m+\pare{1+\chi_0}r}\times \cG_F}}}\pare{t'} \d t' \\[5mm]
\displaystyle \frac{\chi_0}{\pare{1+\chi_0}^2}\int_0^t S_{\frac{1+\chi_0}{\tau}, \sigma}\pare{\partial, t -t'} \set{ \cQ\bra{\Big.\cG_F \times\pare{\big. \pare{m+\pare{1+\chi_0}r}\times \cG_F}}}\pare{t'} \d t'
\end{array}
},
\end{equation}
\begin{equation}\label{eq:cT2}
\cT_2\bra{U} = \pare{
\begin{array}{c}
0 \\[5mm]
\displaystyle \frac{\chi_0}{1+\chi_0} \int_0^t S_{\frac{1}{\tau}, \sigma}\pare{\partial, t -t'} \set{  \cP\bra{\Big.\cG_F \times\pare{\big. r\times m}} +  \cP \bra{\Big. \pare{m+r}\times\bra{\pare{m+\pare{1+\chi_0}r}\times\cG_F}}}\pare{t'} \d t' \\[5mm]
\displaystyle\frac{\chi_0}{1+\chi_0} \int_0^t S_{\frac{1+\chi_0}{\tau}, \sigma}\pare{\partial, t -t'} \set{  \cQ\bra{\Big.\cG_F \times\pare{\big. r\times m}} +  \cQ \bra{\Big. \pare{m+r}\times\bra{\pare{m+\pare{1+\chi_0}r}\times\cG_F}}}\pare{t'} \d t'
\end{array}
},
\end{equation}
\begin{equation}\label{eq:cT3}
\cT_3\bra{U} = \pare{
\begin{array}{c}
0 \\[5mm]
\displaystyle \int_0^t S_{\frac{1}{\tau}, \sigma}\pare{\partial, t -t'} \set{ \cP\bra{\pare{m+r}\times\pare{r\times {m}}\Big. }}\pare{t'} \d t' \\[5mm]
\displaystyle \int_0^t S_{\frac{1+\chi_0}{\tau}, \sigma}\pare{\partial, t -t'} \set{ \cQ\bra{\pare{m+r}\times\pare{r\times {m}}\Big. }}\pare{t'} \d t'
\end{array}
}.
\end{equation}

\noindent
While finally we can define the outer force $ g $ as
\begin{equation}\label{eq:def_outer_g}
g = \pare{
\begin{array}{c}
\displaystyle
\frac{\chi_0}{\pare{1+\chi_0}^2} \int_0^t S_{0, \nu}\pare{\partial, t-t'}\left( \cG_F\cdot\nabla \cG_{F} \right)\pare{t'} \d t' \\[5mm]
0 \\[5mm]
\displaystyle -\frac{\chi_0}{1+\chi_0}\int_0^tS_{\frac{1+\chi_0}{\tau}, \sigma}\pare{\partial, t - t'}\pare{ \partial_t \cG_F - \sigma \Delta \cG_F \Big.}\pare{t'} \d t'
\end{array}
}=\pare{\begin{array}{c}
g_1 \\ 0 \\ g_2
\end{array} }
\end{equation}

Despite the long and tedious computations we can already understand why we decided to rewrite system \eqref{eq:Shilomis_mild} in the abstract form \eqref{eq:Shilomis_mild2}. the integral operators defined explicitly in \eqref{eq:cT2NS}--\eqref{eq:cT3} are all of the following form: a time convolution of a nonlinearity which falls within one of the six cases explained in Remark \ref{rem:explanation_nonlinear_terms} with one operator of the form $ S_{\mu, \gamma} \pare{\partial} $ defined in \eqref{eq:propagator_Sgammamu}.

\subsection{Parabolic estimates for generalized Shilomis-type nonlinearities}\label{sec:par_est_for_S}

It suffice hence to check that the nonlinear integral operator defined by the right hand side of \eqref{eq:Shilomis_mild2} is continuous in $ \LqHu $ in order to apply Proposition \ref{prop:fixed_point} and to deduce the existence of a fixed point for the integral equation \eqref{eq:Shilomis_mild}. Indeed to prove the continuity of each term in the nonlinearity given by \eqref{eq:def_cT} would be a lengthy and tedious work. On the other hand we can exploit the observations deduced in Remark \ref{rem:explanation_nonlinear_terms}: every term appearing in \eqref{eq:nonlinearities_Shilomis_mild} belongs to one of at most six classes of nonlinearities, this significantly simplifies the process.

\begin{prop}\label{prop:nonlinear_bounds_generic}
Let $ v, v_1, v_2, v_3, G \in \LqHu $, let $ \pare{\gamma, \mu} \in \left[0, \infty\right) \times \pare{0, \infty} $ and let $ B_{\ns}, \cL_j, \cN_p, \ j=1, 2, \ p=1, 2, 3 $ be as in Remark \ref{rem:explanation_nonlinear_terms}, then setting $ S_{\gamma, \mu} $ the propagator defined in \eqref{eq:propagator_Sgammamu} the following inequalities hold true
\begin{enumerate}

\item \label{enum:nonlinear_bounds_generic1} $ \displaystyle\norm{S_{\gamma, \mu}\pare{\partial}\star_t B_{\ns}\pare{v_1, v_2}}_{\LqHu}\leqslant \frac{C}{\mu ^{3/4}} \norm{v_1}_{\LqHu}\norm{v_2}_{\LqHu} $, 

\item \label{enum:nonlinear_bounds_generic2} $ \displaystyle\norm{S_{\gamma, \mu}\pare{\partial}\star_t \cL_j \pare{v}}_{\LqHu}\leqslant \frac{C}{\mu ^{3/4}} \norm{G}_{\LqHu}\norm{v}_{\LqHu} $, 

\item \label{enum:nonlinear_bounds_generic3} $ \norm{S_{\gamma, \mu}\pare{\partial}\star_t \cN_p \pare{v_1, \ldots , v_p}}_{\LqHu}\leqslant \displaystyle \frac{C}{\mu ^{1/2}} \pare{ \prod_{i=1}^p \norm{v_i}_{\LqHu}} \times \norm{G}^{3-p}_{\LqHu} $.
\end{enumerate}
\end{prop}

\begin{proof}

\begin{enumerate}

\item Indeed $ S_{\gamma, \mu}\pare{\partial}\star_t B_{\ns}\pare{v_1, v_2} $ can be though as the unique solution of \eqref{eq:parabolic_linear} when $ w_0 = 0 $ and $ F = B_{\ns}\pare{v_1, v_2} $, hence applying Lemma \ref{lem:parabolic1} we deduce that $ \norm{S_{\gamma, \mu}\pare{\partial}\star_t B_{\ns}\pare{v_1, v_2}}_{\LqHu} \leqslant \frac{C}{\mu ^{3/4}} \norm{  B_{\ns}\pare{v_1, v_2} }_{L^2_T\dot{H}^{-1/2}} $. Moreover since every term in $ B_{\ns} $ is of the form $ v_i \ q_{i, j}^{\ns, \ell}\pare{\partial} v_j $ where $  q_{i, j}^{\ns, \ell} $ homogeneous Fourier multiplier of order one applying Lemma \ref{lem:Sob_product_rules} we deduce
\begin{equation*}
\begin{aligned}
\norm{  B_{\ns}\pare{v_1, v_2} }_{L^2_T\dot{H}^{-1/2}} & \leqslant C \pare{ \norm{v_1 \otimes \nabla v_2}_{L^2_T\dot{H}^{-1/2}} + \norm{\nabla v_1 \otimes  v_2}_{L^2_T\dot{H}^{-1/2}}}, \\
&\leqslant C\norm{v_1}_{\LqHu}\norm{v_2}_{\LqHu}, 
\end{aligned}
\end{equation*}
proving the first inequality.

\item Similarly as above $ S_{\gamma, \mu}\pare{\partial}\star_t \cL_j \pare{v} $ is the unique solution of \eqref{eq:parabolic_linear} when $ w_0=0 $ and $ F = \cL_j \pare{v}  $, whence \linebreak$ \norm{S_{\gamma, \mu}\pare{\partial}\star_t  \cL_j \pare{v}}_{\LqHu} \leqslant \frac{C}{\mu ^{3/4}} \norm{   \cL_j \pare{v} }_{L^2_T\dot{H}^{-1/2}} $. Using again Lemma \ref{lem:Sob_product_rules} we deduce
\begin{equation*}
\begin{aligned}
\norm{  \cL_j \pare{v}}_{L^2_T\dot{H}^{-1/2}} & \leqslant C \pare{ \norm{G \otimes \nabla v}_{L^2_T\dot{H}^{-1/2}} + \norm{\nabla G \otimes  v}_{L^2_T\dot{H}^{-1/2}}}, \\
&\leqslant C \norm{G }_{\LqHu}\norm{v}_{\LqHu}, 
\end{aligned}
\end{equation*}
concluding the proof of the second inequality. 

\item Similarly as above we can deduce the estimate
\begin{equation*}
 \norm{S_{\gamma, \mu}\pare{\partial}\star_t  \cN_p \pare{v_1, \ldots , v_p}}_{\LqHu} \leqslant \frac{C}{\mu ^{1/2}} \norm{   \cN_p \pare{v_1, \ldots , v_p} }_{L^{4/3}_TL^2} , 
\end{equation*}
using Lemma \ref{lem:parabolic2}. Whence, since
\begin{equation*}
\cN_p \pare{v_1, \ldots , v_p} \sim v^{\otimes p} \otimes G^{\otimes\pare{3-p}}, 
\end{equation*}
using repeatedly H\"older inequality and the continuous embedding $ \dot{H}^1\hra L^6 $ we deduce
\begin{equation*}
\norm{   \cN_p \pare{v_1, \ldots , v_p} }_{L^{4/3}_TL^2}  \leqslant \pare{ \prod_{i=1}^p \norm{v_i}_{\LqHu}} \times \norm{G}^{3-p}_{\LqHu}. 
\end{equation*}
\end{enumerate}
\end{proof}

\subsection{Bounds for the system \eqref{eq:Shilomis_mild2}} \label{sec:nonlinear_bounds_Shilomis}

As mentioned above the scope of the present section is to apply the nonlinear bounds proved in Section \ref{sec:nonlinear_bounds_Shilomis} to the Shilomis system in mild form \eqref{eq:Shilomis_mild2}. Such bounds will be provided systematically in the present section. \\

At first we need to estimate the contributions provided by the initial datum:
\begin{prop}\label{prop:smallness_initial_data}
Let $ u_0\in \Hud, m_0, r_0\in \dot{H}^1 $, then
\begin{enumerate}

\item $ \displaystyle\norm{S_{0, \nu}\pare{\partial} u_0}_{\LqHu} \leqslant\frac{C}{\nu^{1/4}} \norm{u_0}_{\Hud} $, 

\item $ \displaystyle\norm{S_{\frac{1}{\tau}, \sigma}\pare{\partial} m_0}_{\LqHu} \leqslant {C}{\tau^{1/4}} \norm{m_0}_{\dot{H}^{1}} $, 

\item $ \displaystyle\norm{S_{\frac{1+\chi_0}{\tau}, \sigma}\pare{\partial} r_0}_{\LqHu} \leqslant \frac{C\tau^{1/4}}{\pare{1+\chi_0}^{1/4}} \norm{r_0}_{\dot{H}^{1}} $.

\end{enumerate}
\end{prop}

\begin{proof}
We apply respectively Lemma \ref{lem:parabolic1} to deduce the first inequality and Lemma \ref{lem:parabolic2} to deduce the second and the third inequality. 
\end{proof}

Next we  bound the bulk force:

\begin{prop}\label{prop:smallness_outer_force}
Let $ \cG_F\in \LqHu \cap L^2_T\dot{H}^{3} \cap \dot{W}^{1, 2}_T \dot{H}^1 $, and let $ \tau $ be
\begin{equation}\label{eq:hyp_smallness_tau1}
\tau < \frac{\pare{1+\chi_0}^{7/3}}{ C\ \chi_0^{4/3}} \pare{\norm{\cG_F}_{L^2_T\dot{H}^{3}}   + \norm{\cG_F}_{\dot{W}^{1, 2}_T \dot{H}^1}}^{-4/3} \ \varrho_0^{4/3},
\end{equation}
 then let us consider $ g $ defined as in \eqref{eq:def_outer_g}, the following bound is true
 \begin{equation*}
 \norm{g}_{\LqHu}\leqslant \varrho_0. 
 \end{equation*}

\end{prop}

\begin{proof}
Let us define 
\begin{align*}
f_1 & = \frac{\chi_0}{\pare{1+\chi_0}^2} \cG_F\cdot\nabla \cG_{F}, \\
f_2 & = -\frac{\chi_0}{1+\chi_0}\pare{ \partial_t \cG_F - \sigma \Delta \cG_F \Big.}. 
\end{align*}
And let $ g_1, g_2 $ be defined as in \eqref{eq:def_outer_g}, indeed
\begin{align*}
g_1 & = S_{0, \nu}\pare{\partial}\star_t f_1, \\
g_2 & = S_{\frac{1+\chi_0}{\tau}, \sigma}\pare{\partial}\star_t f_2, \\
\end{align*}

Indeed $ g_1 $ is the unique solution of the following Cauchy problem
\begin{equation*}
\left\lbrace
\begin{aligned}
& \partial_t w -\nu\Delta w = f_1, \\
& \left. w\right|_{t=0}=0, 
\end{aligned}
\right. 
\end{equation*}
whence applying Lemma \ref{lem:parabolic1} we deduce
\begin{equation*}      
\norm{g_1}_{\LqHu}\leqslant \frac{C \ \chi_0}{\pare{1+\chi_0}^2\ \nu^{3/4}}\norm{\cG_F\cdot \nabla \cG_{F}}_{L^2_T\dot{H}^{-\frac{1}{2}}}. 
\end{equation*}
Lemma \ref{lem:Sob_product_rules} and the fact that $ \norm{\cG_F}_{\LqHu}\leqslant \varrho_0 $ imply that
\begin{equation}\label{eq:bound_g1}
\norm{g_1}_{\LqHu}\leqslant \frac{C\chi_0 \varrho_0^2}{\pare{1+\chi_0}^2\nu^{3/4}}. 
\end{equation}
In order to bound $ g_2 $ we apply Lemma \ref{lem:smoothin_bulk_force} with $ w_0=0 $ and $ F=f_2 $, obtaining the bound
\begin{equation*}
\norm{g_2}_{\LqHu} =  
 \norm{S_{\frac{1+\chi_0}{\tau}, \sigma}\pare{\partial} \star_t f_2}_{\LqHu}\leqslant \frac{C\chi_0 \tau^{3/4}}{\pare{1+\chi_0}^{7/4}} \pare{\norm{\cG_F}_{L^2_T\dot{H}^{3}}   + \norm{\cG_F}_{\dot{W}^{1, 2}_T \dot{H}^1}},
\end{equation*}
which with the bound \eqref{eq:hyp_smallness_tau1} implies that
\begin{equation}\label{eq:bound_g2}
\norm{g_2}_{\LqHu}\leqslant \frac{\varrho_0}{2}. 
\end{equation}
If $\displaystyle \varrho_0 \leqslant \frac{\pare{1+\chi_0}^2\nu^{3/4}}{2 C \chi_0} $ the bounds \eqref{eq:bound_g1} and \eqref{eq:bound_g2} imply that
\begin{equation*}
\norm{g}_{\LqHu}\leqslant {\varrho_0}. 
\end{equation*}
\end{proof}

We prove now the nonlinear bounds; in order to do so we need to explicit the time-convolution form of the nonlinearities defined in \eqref{eq:cT2NS}--\eqref{eq:cT3},  let us hence define
\begin{equation}\label{eq:cB2NS}
B_{ \ns} \pare{U, U} =
\pare{\begin{array}{c}
B_{ \ns, u} \pare{U, U} \\
B_{ \ns, m} \pare{U, U} \\
B_{ \ns, r} \pare{U, U} 
\end{array}}
=
 \pare{
\begin{array}{c}
\displaystyle - \cP \pare{\big.u\cdot\nabla u } - \cP\bra{ \big. \pare{m+r }\cdot \nabla {r }}
 			- \frac{1}{2}\cP \ \curl \bra{\big.\pare{m+r } \times {r }}\\[5mm]
 			\displaystyle -\cP \bra{  \big. u\cdot\nabla \pare{ m+r }} + \frac{1}{2}\cP \bra{ \big. \pare{\curl \ u} \times \pare{m+r }  }  \\[5mm]
 		\displaystyle  -\cQ \bra{u\cdot\nabla \pare{ m+r }} + \frac{1}{2}\cQ \bra{ \pare{\curl \ u} \times \pare{m+r    }}
\end{array}
},
\end{equation}

\begin{equation}
\cL_{1}\pare{U}=
\pare{
\begin{array}{c}
\cL_{1, u}\pare{U} \\
\cL_{1, m}\pare{U} \\
\cL_{1, r}\pare{U} 
\end{array}
}
 = \pare{
\begin{array}{c}
\displaystyle \frac{1}{2 \pare{ 1+\chi_0}} \left( \Big.  
\cP\bra{ \pare{m+\pare{ 1-\chi_0}r }\cdot \nabla {  \cG_F}}
 			+ \cP  \bra{\pare{m+\pare{1+\chi_0}r}\div\ \cG_F}  \right)\\[5mm]
 			\displaystyle-\frac{\chi_0}{1+\chi_0}\cP\bra{\big. u\cdot\nabla \cG_F} \\[5mm]
 			\displaystyle-\frac{\chi_0}{1+\chi_0}\cQ \bra{\big. u\cdot\nabla \cG_F} 
\end{array}
},
\end{equation}

\begin{equation}
\cL_{2}\pare{U}=
\pare{
\begin{array}{c}
\cL_{2, u}\pare{U} \\
\cL_{2, m}\pare{U} \\
\cL_{2, r}\pare{U} 
\end{array}
}
  = \pare{
\begin{array}{c}
\displaystyle \frac{1}{2 \pare{ 1+\chi_0}}  \pare{\Big. -\cG_{F} \div\pare{m+\pare{1+\chi_0}r} -\cG_{F}\cdot \nabla\pare{m+\pare{1+\chi_0}r}}\\[5mm]
 			\displaystyle \frac{\chi_0}{2\pare{ 1+\chi_0}}\Big. \cP\bra{\curl u \times \cG_F} \\[5mm]
 			\displaystyle\frac{\chi_0}{2\pare{ 1+\chi_0}}\Big. \cQ\bra{\curl u \times \cG_F}' 
\end{array} 
 }. 
\end{equation}

\begin{equation}
\cN_1\pare{U} 
= \pare{
\begin{array}{c}
\cN_{1, u}\pare{U}\\
\cN_{1, m}\pare{U}\\
\cN_{1, r}\pare{U}
\end{array}
}
= \pare{
\begin{array}{c}
0 \\[5mm]
\displaystyle \frac{\chi_0}{\pare{1+\chi_0}^2} \set{ \cP\bra{\Big.\cG_F \times\pare{\big. \pare{m+\pare{1+\chi_0}r}\times \cG_F}}} \\[5mm]
\displaystyle \frac{\chi_0}{\pare{1+\chi_0}^2} \set{ \cQ\bra{\Big.\cG_F \times\pare{\big. \pare{m+\pare{1+\chi_0}r}\times \cG_F}}}
\end{array}
},
\end{equation}

\begin{equation}
\cN_2\pare{U, U}=
\pare{
\begin{array}{c}
\cN_{2, u}\pare{U, U}\\
\cN_{2, m}\pare{U, U}\\
\cN_{2, r}\pare{U, U}
\end{array}
}
 = \pare{
\begin{array}{c}
0 \\[5mm]
\displaystyle  \frac{\chi_0}{1+\chi_0} \cP\bra{\Big.\cG_F \times\pare{\big. r\times m}} + \frac{\chi_0}{1+\chi_0} \cP \bra{\Big. \pare{m+r}\times\bra{\pare{m+\pare{1+\chi_0}r}\times\cG_F}}\\[5mm]
\displaystyle  \frac{\chi_0}{1+\chi_0} \cQ\bra{\Big.\cG_F \times\pare{\big. r\times m}} + \frac{\chi_0}{1+\chi_0} \cQ \bra{\Big. \pare{m+r}\times\bra{\pare{m+\pare{1+\chi_0}r}\times\cG_F}}
\end{array}
},
\end{equation}

\begin{equation}\label{eq:cN3}
\cN_3\pare{U, U, U}=
\pare{
\begin{array}{c}
\cN_{3, u}\pare{U, U, U} \\
\cN_{3, m}\pare{U, U, U} \\
\cN_{3, r}\pare{U, U, U} 
\end{array}
}
 = \pare{
\begin{array}{c}
0 \\[5mm]
\displaystyle  \cP\bra{\pare{m+r}\times\pare{r\times {m}}\Big. } \\[5mm]
\displaystyle  \cQ\bra{\pare{m+r}\times\pare{r\times {m}}\Big. }
\end{array}
}.
\end{equation}

With such notation we can rewrite the operators defined in \eqref{eq:cT2NS}--\eqref{eq:cT3} in a time-convolution form
\begin{equation}
\label{eq:nonlin_as_convolutions}
\begin{aligned}
\cT_{2, \ns}\bra{U}& =  
\pare{\begin{array}{c}
S_{0, \nu}\pare{\partial} \star_t  B_{ \ns, u} \pare{U, U} \\
S_{\frac{1}{\tau}, \sigma} \pare{\partial} \star_t B_{ \ns, m} \pare{U, U} \\
S_{\frac{1+\chi_0}{\tau}, \sigma} \pare{\partial} \star_t B_{ \ns, r} \pare{U, U} 
\end{array}},&
\cT_{1, \RN{1}} \bra{U} & = \pare{
\begin{array}{c}
S_{0, \nu}\pare{\partial} \star_t\cL_{1, u}\pare{U} \\
S_{\frac{1}{\tau}, \sigma} \pare{\partial} \star_t \cL_{1, m}\pare{U} \\
S_{\frac{1+\chi_0}{\tau}, \sigma} \pare{\partial} \star_t \cL_{1, r}\pare{U} 
\end{array}
}, \\
\cT_{1, \RN{2}} \bra{U} & = \pare{
\begin{array}{c}
S_{0, \nu}\pare{\partial} \star_t\cL_{2, u}\pare{U} \\
S_{\frac{1}{\tau}, \sigma} \pare{\partial} \star_t \cL_{2, m}\pare{U} \\
S_{\frac{1+\chi_0}{\tau}, \sigma} \pare{\partial} \star_t \cL_{2, r}\pare{U} 
\end{array}
}, &
\cT_{1} \bra{U} & = 
\pare{
\begin{array}{c}
S_{0, \nu}\pare{\partial} \star_t \cN_{1, u}\pare{U}\\
S_{\frac{1}{\tau}, \sigma} \pare{\partial} \star_t \cN_{1, m}\pare{U}\\
S_{\frac{1+\chi_0}{\tau}, \sigma} \pare{\partial} \star_t  \cN_{1, r}\pare{U}
\end{array}
},\\
\cT_{2}\bra{U} & = 
\pare{
\begin{array}{c}
S_{0, \nu}\pare{\partial} \star_t \cN_{2, u}\pare{U, U}\\
S_{\frac{1}{\tau}, \sigma} \pare{\partial} \star_t \cN_{2, m}\pare{U, U}\\
S_{\frac{1+\chi_0}{\tau}, \sigma} \pare{\partial} \star_t \cN_{2, r}\pare{U, U}
\end{array}
}, &
\cT_{3}\bra{U} & = 
\pare{
\begin{array}{c}
S_{0, \nu}\pare{\partial} \star_t \cN_{3, u}\pare{U, U, U} \\
S_{\frac{1}{\tau}, \sigma} \pare{\partial} \star_t \cN_{3, m}\pare{U, U, U} \\
S_{\frac{1+\chi_0}{\tau}, \sigma} \pare{\partial} \star_t \cN_{3, r}\pare{U, U, U} 
\end{array}
} .
\end{aligned}
\end{equation}

It is hence not a coincidence that the nonlinearities in \eqref{eq:cB2NS}--\eqref{eq:cN3} have the same notation as the nonlinearities on which we provide the bounds in Section \ref{sec:par_est_for_S}, setting in fact $ G=\cG $ and $ U=\pare{u, m, r} $ we can express the nonlinearity $ \cT\bra{U} $ of \eqref{eq:Shilomis_mild2} in the form \eqref{eq:nonlin_as_convolutions} we can use the results of Section \ref{sec:par_est_for_S} in order to prove th following result:

\begin{prop}\label{prop:nonlinear_bounds_Shilomis}
Let $ c = \min\set{\big. \nu, \sigma} $, and let $ \norm{\cG_{F}}_{\LqHu}\leqslant \varrho_0 $, then the following bounds hold true
\begin{enumerate}
\item $ \displaystyle \norm{\cT_{1, j }\bra{U}}_{\LqHu}\leqslant \frac{C}{c^{3/4}}\ \varrho_0 \norm{U}_{\LqHu} $ for $ j=\RN{1}, \RN{2} $, 

\item $ \displaystyle \norm{\cT_{2, \ns}\bra{U}}_{\LqHu}\leqslant\frac{C}{c^{3/4}}  \norm{U}_{\LqHu}^2 $, 

\item $ \displaystyle \norm{\cT_{p}\bra{U}}_{\LqHu}\leqslant\frac{C}{c^{1/2}}\ \varrho_0^{3-p}  \norm{U}_{\LqHu}^p $ for $ p=1,2,3 $. 

\end{enumerate}
\end{prop}

\begin{proof}
thanks to the results stated and proved in Section \ref{sec:par_est_for_S} the proof of Proposition \ref{prop:nonlinear_bounds_Shilomis} is now immediate. 
\begin{enumerate}

\item We know that $ \cT_{1, j} $ can be written in convolution form as it is done in \eqref{eq:nonlin_as_convolutions}, whence we use the estimates proved in Proposition \ref{prop:nonlinear_bounds_generic}, \ref{enum:nonlinear_bounds_generic2} to deduce the bound
\begin{equation*}
\norm{\cT_{1, j }\bra{U}}_{\LqHu}\leqslant \frac{C}{c^{3/4}} \norm{\cG_{F}}_{\LqHu}\norm{U}_{\LqHu}, 
\end{equation*}
but since $ \norm{\cG_{F}}_{\LqHu}\leqslant \varrho_0 $ we deduce the first bound. 

\item Similarly as before we exploit the convolution formulation of $ \cT_{2, \ns} $ given in \eqref{eq:nonlin_as_convolutions} and we use the bound proved in Proposition \ref{prop:nonlinear_bounds_generic}, \ref{enum:nonlinear_bounds_generic1} to deduce
\begin{equation*}
\norm{\cT_{2, \ns}\bra{U}}_{\LqHu}\leqslant\frac{C}{c^{3/4}}  \norm{U}_{\LqHu}^2. 
\end{equation*}

\item As in the first two steps, but using the bound proved in Proposition \ref{prop:nonlinear_bounds_generic}, \ref{enum:nonlinear_bounds_generic3}, we deduce that
\begin{equation*}
\norm{\cT_{p}\bra{U}}_{\LqHu}\leqslant\frac{C}{c^{1/2}}\ \norm{\cG}_{\LqHu}^{3-p}  \norm{U}_{\LqHu}^p, \hspace{5mm}\text{for }p=1,2,3, 
\end{equation*}
but again since $ \norm{\cG_{F}}_{\LqHu}\leqslant \varrho_0 $ we prove the last bound. 

\end{enumerate}
\end{proof}

\subsection{The fixed point theorem } \label{sec:fixed_poin_application}

We can at this point apply Proposition \ref{prop:fixed_point} to the system \ref{eq:Shilomis_mild2}. Let us define

\begin{equation*}
y  = \cS\pare{\partial} U_0 + g, 
\end{equation*}
where $ \cS\pare{\partial} U_0 $ is defined in \eqref{eq:cSU0} and $ g $ is defined in \eqref{eq:def_outer_g}. Next let us define
\begin{equation*}
T_1\pare{U} = \cT_{1, \RN{1}}\bra{U} + \cT_{1, \RN{2}}\bra{U} + \cT_{1}\bra{U}, 
\end{equation*}
where $ \cT_{1, \RN{1}}, \cT_{1, \RN{2}} $ and $ \cT_{1} $ are respectively defined in \eqref{eq:cT1I}, \eqref{eq:cT1II} and \eqref{eq:cT1}. Next
\begin{equation*}
T_2\pare{U, U} = \cT_{2, \ns}\bra{U} + \cT_2\bra{U}, 
\end{equation*}
where $ \cT_{2, \ns}, \cT_2 $ are defined in \eqref{eq:cT2NS} and \eqref{eq:cT2}. Finally we define
\begin{equation*}
T_3\pare{U, U, U} = \cT_3\bra{U}. 
\end{equation*}

In order to apply Proposition \ref{prop:fixed_point} we have to check the following three conditions
\begin{enumerate}[i]

\item The element $ y=\cS\pare{\partial} U_0 + g $ belongs to the ball $ B_{\LqHu}\pare{0, \rho } $ for $ \rho $ small, 

\item Each $ p $--linear operator $ T_p, \ p=1, 2, 3 $ maps continuously $ \pare{\LqHu}^p $ to $ \LqHu $, 

\item The norm of $ T_1 $ as a linear operator from $ \LqHu $ to itself is strictly smaller than $ 1/4 $. 

\end{enumerate}

We prove hence these conditions here below;

\begin{enumerate}[i]

\item A standard triangular inequality tells us that
\begin{equation*}
\norm{y}_{\LqHu} \leqslant \norm{\cS\pare{\partial} U_0}_{\LqHu} + \norm{g}_{\LqHu}, 
\end{equation*}
whence, thanks to the results proved in Proposition \ref{prop:smallness_initial_data} we can argue that if 
\begin{align*}
\norm{u_0}_{\Hud} \leqslant \frac{\nu^{1/4}}{6C} \ \rho, &&
\tau < \frac{1+\chi_0}{6C^4\pare{\norm{m_0}_{\dot{H}^1}^4+\norm{r_0}_{\dot{H}^1}^4}}\ \rho^4, 
\end{align*}
then 
\begin{equation*}
 \norm{\cS\pare{\partial} U_0}_{\LqHu} \leqslant \frac{\rho}{2}. 
\end{equation*}
While if $ \varrho_0 < \frac{\rho}{2} $ Proposition \ref{prop:smallness_outer_force} assures us that $ \norm{g}_{\LqHu}<\rho/2 $, proving the first claim. 

\item Proposition \ref{prop:nonlinear_bounds_Shilomis} assures us that each $ p $--linear operator $ T_p, \ p=1, 2, 3 $ maps continuously $ \pare{\LqHu}^p $ to $ \LqHu $. 

\item We use again the result in Proposition \ref{prop:nonlinear_bounds_Shilomis} to deduce that
\begin{equation*}
\norm{T_1}\leqslant \frac{2C\pare{1+\varrho_0}}{\min \set{c^{1/2},\big.  c^{3/4}}} \ \varrho_0, 
\end{equation*}
whence we deduce that if 
\begin{equation*}
\varrho_0 <\frac{\min\set{c^{1/2},\big.  c^{3/4}}}{8C}, 
\end{equation*}
then $ \norm{T_1}<1/4 $. 

\end{enumerate}

We can hence apply Proposition \ref{prop:fixed_point} to deduce the existence a unique solution to the equation in mild form \eqref{eq:Shilomis_mild2}, which in turn implies the existence of a unique solution $ \pare{u, m, r}\in \LqHu $ to \eqref{eq:Shilomis2}. \\

The continuity w.r.t. the $ \Hud $ topology, i.e. that $ U\in \cC_T\Hud $, follows from standard considerations which are analogous  to the incompressible \NS case, see \cite{LR2}. 
\hfill $ \Box $

\section{Convergence as $ \tau\to 0 $}\label{sec:conv_tau}

In the previous section we proved that it is possible to construct solutions of \eqref{eq:Shilomis2} in a critical functional space \textit{independently} of the parameter $ \tau $, when $ \tau $ is sufficiently small. In the present section we let $ \tau\to 0 $ and we deduce the limit system solved by $ \pare{u^\tau, m^\tau, r^\tau} $ in the limit $ \tau\to 0 $. Just for this section, since we are interested to compute the asymptotic as $ \tau\to 0 $, we explicit the dependence of the unknown on the parameter $ \tau $.  The result we prove is the following one. 

\begin{prop}\label{prop:convergence}
Let $ \pare{u_0, m_0, r_0}, \cG_F $ and $ \tau $ be as in the statement \ref{en:point_we_prove_in_the_ecistence_thm} of Proposition \ref{prop:existence_unique_solution}, and let us moreover assume that $ \nabla \cG_F \in L^2_T\Hud$. Then for any $ \varepsilon > 0 $
\begin{equation}\label{eq:conv_mr_to_0}
\norm{\pare{m^\tau, r^\tau}}_{L^\infty\pare{\pare{\varepsilon, T};\Hud}}\xrightarrow{\tau\to 0}0.  
\end{equation}
Moreover for each $ t\in\bra{0, T} $ the following energy bound holds true
\begin{equation}\label{eq:energy_Hud_est_mr}
\frac{1}{2}\norm{\pare{m^\tau\pare{t}, r^\tau \pare{t}}}_{\Hud}^2 + \frac{1}{\tau}\int_0^t \norm{\pare{m^\tau\pare{t'}, r^\tau \pare{t'}}}_{\Hud}^2 \d t' + \sigma \int_0^t \norm{\pare{\nabla m^\tau\pare{t'}, \nabla r^\tau \pare{t'}}}_{\Hud}^2 \d t'\leqslant \frac{C}{\sigma}\ \rho^4, 
\end{equation}
where $ \rho $ is the radius of the ball in which the solutions constructed in Proposition \ref{prop:existence_unique_solution} live. \\
Moreover $ u^\tau\xrightarrow{\tau\to 0}\bu $ in $ L^\infty\pare{\pare{\varepsilon, T};\Hud}$ and $ \nabla u^\tau\xrightarrow{\tau\to 0} \nabla \bu $ in $  L^2\pare{\pare{\varepsilon, T};\Hud} $, where $ \bu $ is the solution of the following incompressible \NS equations 
\begin{equation}\label{eq:limit_system}
\left\lbrace
\begin{aligned}
& \partial_t\bu + \bu\cdot\nabla\bu -\nu \Delta\bu +\nabla\bar{p} = \frac{\chi_0}{\pare{1+\chi_0}^2}\ \cG_F\cdot\nabla\cG_F, \\
& \div\ \bu=0, \\
&\left.\bu\right|_{t=0} = u_0.  
\end{aligned}
\right.
\end{equation}
\end{prop}

\begin{rem}
\begin{itemize}

\item We want to point out that the systems \eqref{eq:limit_system_thm} and \eqref{eq:limit_system} are equivalent. Indeed since $ \cG_F = \nabla \Delta^{-1} F $ it is not difficult to deduce that
\begin{equation*}
\cG_F \cdot \nabla \cG_F = \frac{1}{2} \ \nabla \av{\nabla\Delta^{-1} F}^2. 
\end{equation*}

\item Thanks to the result proved in Proposition \ref{prop:existence_unique_solution} it is not surprising, performing an energy estimate, to deduce that\footnote{See the energy estimate \eqref{eq:energy_Hud_est_mr} and its proof for a complete argument. }
\begin{equation*}
\norm{\pare{r^\tau, m^\tau}}_{L^2_T\Hud} = \cO\pare{\tau}, \text{ as } \tau\to 0.
\end{equation*}
Unfortunately such convergence is not strong enough in order to deduce that $ u^\tau $ converges toward $ \bu $ solution of \eqref{eq:limit_system} in the critical topology $  L^\infty\pare{\pare{\varepsilon, T};\Hud}\cap  L^2\pare{\pare{\varepsilon, T};\dot{H}^{\frac{3}{2}}}  $ (it though sufficient in order to deduce that there is convergence in some weak sense). We must therefore prove that $ m^\tau, r^\tau $ converge to zero in a stronger topology in order to prove convergence in critical norms, for this reason we have to prove the particular convergence stated in  \eqref{eq:conv_mr_to_0}.

\end{itemize}
 \fine
\end{rem}

\begin{proof}
We will divide the proof of Proposition \ref{prop:convergence} in several steps
\begin{enumerate}[\bf Step 1 :]

\item Proof of \eqref{eq:conv_mr_to_0}. \\
We prove the result for $ m^\tau $ only being the procedure for $ r^\tau $ identical. 
Let us rewrite the evolution equation of $ m^\tau $, given in \eqref{eq:Shilomis2}, as 
\begin{equation}\label{eq:simplified_equation_m}
\partial_t m^\tau + \frac{1}{\tau}\ m^\tau - \sigma\Delta m^\tau = F_1^\tau + F_2^\tau, 
\end{equation}
where
\begin{equation*}
\begin{aligned}
F_1^\tau & = -\cP \bra{u^\tau\cdot\nabla \pare{ m^\tau+r^\tau + \frac{\chi_0}{1+\chi_0} \cG_F }} + \frac{1}{2}\cP \bra{ \pare{\curl \ u^\tau} \times \pare{m^\tau+r^\tau + \frac{\chi_0}{1+\chi_0} \cG_F}  } \\
		F_2^\tau & =	- \cP \set{ \pare{m^\tau+r^\tau + \frac{\chi_0}{1+\chi_0} \cG_F} \times \bra{ \pare{m^\tau+r^\tau + \frac{\chi_0}{1+\chi_0} \cG_F} \times \pare{-r^\tau + \frac{1}{1+\chi_0} \cG_F}} }.
\end{aligned}
\end{equation*}
Hence thanks to the result proved in Proposition \ref{prop:existence_unique_solution} we know that there exists a $ \tau_0 = \tau_0\pare{u_0,m_0,r_0} $ and a $ T\in \left(0, \infty\right] $ so that $ \pare{u^\tau, m^\tau, r^\tau}\in \LqHu $ uniformly for $ \tau\in\bra{0, \tau_0} $. Hence since by hypothesis $ \cG_F\in\LqHu $ we deduce that
\begin{align*}
F_1^\tau \in L^2_T \dot{H}^{-1/2}, && F_2^\tau\in L^{4/3}_T L^2, 
\end{align*}
uniformly for $ \tau\in\bra{0, \tau_0} $.\\
We can hence apply the estimate \eqref{eq:linear_damping_estimate} of Lemma \ref{lem:linear_damping_estimate} setting $ \gamma=\tau^{-1}, \mu=\sigma $ and $ w=m^\tau $ we deduce 
\begin{equation*}
\norm{m^\tau\pare{t}}_{ \Hud }\leqslant C \pare{
e^{-\frac{t}{\tau}} \norm{m_0}_{\Hud} + \tau^{1/8}\bra{ \norm{F_1^\tau}_{L^2_T \dot{H}^{-1/2}} + \norm{F_2^\tau}_{L^{4/3}_T L^2} } + \frac{1}{\sigma^{1/4}} o_{\frac{1}{\tau}}\pare{1}
},
\end{equation*}
which indeed proves the statement \eqref{eq:conv_mr_to_0} for $ m^\tau $. With the very same procedure we can prove the bound
\begin{equation*}
\norm{r^\tau\pare{t}}_{ \Hud }\leqslant C \pare{
e^{-\frac{1+\chi_0}{\tau} t} \norm{m_0}_{\Hud} + \frac{\tau^{1/8}}{\pare{1+\chi_0}^{1/8}}\bra{ \norm{H_1^\tau}_{L^2_T \dot{H}^{-1/2}} + \norm{H_2^\tau}_{L^{4/3}_T L^2} } + \frac{1}{\sigma^{1/4}} o_{\frac{1+\chi_0}{\tau}}\pare{1}
},
\end{equation*} 
where
\begin{equation*}
\begin{aligned}
H_1^\tau & = -\cQ \bra{u^\tau\cdot\nabla \pare{ m^\tau+r^\tau + \frac{\chi_0}{1+\chi_0} \cG_F }} + \frac{1}{2}\cQ \bra{ \pare{\curl \ u^\tau} \times \pare{m^\tau+r^\tau + \frac{\chi_0}{1+\chi_0} \cG_F}  } \\
		H_2^\tau & =	- \cQ \set{ \pare{m^\tau+r^\tau + \frac{\chi_0}{1+\chi_0} \cG_F} \times \bra{ \pare{m^\tau+r^\tau + \frac{\chi_0}{1+\chi_0} \cG_F} \times \pare{-r^\tau + \frac{1}{1+\chi_0} \cG_F}} },
\end{aligned}
\end{equation*} 
which concludes the proof of \eqref{eq:conv_mr_to_0}.

\item Proof of \eqref{eq:energy_Hud_est_mr}. \\
We prove the bound for $ m^\tau $ only being the procedure for $ r^\tau $ identical. Let us multiply the equation \eqref{eq:simplified_equation_m} for $ \sqrt{-\Delta}\ m^\tau $ and let us integrate in space, integrating by parts if it may be, we deduce the energy inequality
\begin{equation*}
\frac{1}{2}\frac{\d}{\d t}\norm{ {m^\tau}}_{\Hud}^2 + \frac{1}{\tau} \norm{ {m^\tau}}_{\Hud}^2  + \sigma  \norm{ {\nabla m^\tau}}_{\Hud}^2 \leqslant \psc{F^\tau_1}{\sqrt{-\Delta}\ m^\tau}_{L^2\times L^2} + \psc{F^\tau_2}{\sqrt{-\Delta}\ m^\tau}_{L^2\times L^2}.
\end{equation*}
But indeed
\begin{align*}
\psc{F^\tau_1}{\sqrt{-\Delta}\ m^\tau}_{L^2\times L^2} & \leqslant \frac{C}{\sigma} \norm{F^\tau_1}_{\dot{H}^{-\frac{1}{2}}}^2 + \frac{\sigma}{2}\norm{\nabla m^\tau}_{\Hud}^2, \\
\psc{F^\tau_2}{\sqrt{-\Delta}\ m^\tau}_{L^2\times L^2}& \leqslant \norm{F^\tau_2}_{L^2}\norm{m^\tau}_{\dot{H}^1}, 
\end{align*}
whence we deduce
\begin{equation*}
\frac{1}{2}\frac{\d}{\d t}\norm{ {m^\tau}}_{\Hud}^2 + \frac{1}{\tau} \norm{ {m^\tau}}_{\Hud}^2  + \frac{\sigma}{2}  \norm{ {\nabla m^\tau}}_{\Hud}^2 \leqslant \frac{C}{\sigma} \norm{F^\tau_1}_{\dot{H}^{-\frac{1}{2}}}^2 + \norm{F^\tau_2}_{L^2}\norm{m^\tau}_{\dot{H}^1},
\end{equation*}
therefore integrating in time 
\begin{equation*}
\frac{1}{2}\ \norm{ {m^\tau\pare{t}}}_{\Hud}^2 + \frac{1}{\tau} \int_0^t \norm{ {m^\tau\pare{t'}}}_{\Hud}^2\d t'  + \frac{\sigma}{2} \int_0^t  \norm{ {\nabla m^\tau\pare{t'}}}_{\Hud}^2 \d t' \leqslant \frac{C}{\sigma} \norm{F^\tau_1}_{L^2_T\dot{H}^{-\frac{1}{2}}}^2 + \norm{F^\tau_2}_{L^{4/3}_T L^2}\norm{m^\tau}_{L^4_T\dot{H}^1}.
\end{equation*}
It is hence easy to deduce using Lemma \ref{lem:Sob_product_rules} that (here we denote $ U^\tau=\pare{u^\tau, m^\tau, r^\tau} $)
\begin{equation*}
\norm{F^\tau_1}_{L^2_T\dot{H}^{-\frac{1}{2}}}\leqslant \norm{U^\tau}_{\LqHu}\pare{\norm{U^\tau}_{\LqHu} + \norm{\cG_F}_{\LqHu}}, 
\end{equation*}
for the construction given in Proposition \ref{prop:existence_unique_solution} we know that $ \norm{U^\tau}_{\LqHu}\lesssim \rho $, and moreover $ \norm{\cG_F}_{\LqHu}\leqslant\varrho_0 < \rho $ by hypothesis, hence we deduce
\begin{equation*}
\norm{F^\tau_1}_{L^2_T\dot{H}^{-\frac{1}{2}}}^2\leqslant C\rho^4. 
\end{equation*}
Similar computations lead us to deduce the bound $ \norm{F^\tau_2}_{L^{4/3}_T L^2}\lesssim \rho^3 $, whence we conclude the proof of the estimate \eqref{eq:energy_Hud_est_mr}.

\item Convergence toward the limit system \eqref{eq:limit_system}. \\
Indeed under the smallness hypothesis on $ \cG_F $ and $ \bu $ stated in the point \ref{en:point_we_prove_in_the_ecistence_thm} of Proposition \ref{prop:existence_unique_solution} there exists a unique $ \bu $ bar solution of \eqref{eq:limit_system} in the space\footnote{Let us remark that if $ T=\infty $ the solution is global.} $ \cC_T\Hud\cap  \LqHu $. Let us now  select a $ \varepsilon\in\pare{0, T} $ so that
\begin{equation*}
\norm{u^{\tau}\pare{\cdot, \varepsilon} - \bu\pare{\cdot, \varepsilon}}_{\Hud}\leqslant \eta_{\varepsilon}, 
\end{equation*}
where $  \eta_\varepsilon\xrightarrow{\varepsilon\to 0}0 $ since the applications $ t\mapsto\norm{u^\tau\pare{\cdot, t}}_{\Hud} $ and $ t\mapsto\norm{\bu \pare{\cdot, t}}_{\Hud} $ are continuous. 
  Next let us denote as 
\begin{equation*}
\delta u^\tau = u^\tau -\bu, 
\end{equation*}
by the aid of \eqref{eq:Shilomis2} and \eqref{eq:limit_system} we can compute the evolution equation satisfied by $ \delta u^\tau $, i.e.
\begin{equation*}
\partial_t  \delta u^\tau  - \nu \Delta  \delta u^\tau  + \nabla\delta p^\tau = -  \delta u^\tau \cdot\nabla u^\tau + \bu\cdot\nabla  \delta u^\tau + G^\tau, 
\end{equation*}
where the outer force $ G^\tau $ is defined as 
\begin{multline}\label{eq:Gtau}
G^\tau = 
\pare{m^\tau+r^\tau + \frac{\chi_0}{1+\chi_0} \cG_F}\cdot \nabla \pare{-r^\tau + \frac{1}{1+\chi_0} \cG_F} - 
 \frac{\chi_0}{\pare{1+\chi_0}^2} \cG_F\cdot \nabla \cG_F 
 \\
 			+ \frac{1}{2}\curl \bra{\pare{m^\tau+r^\tau + \frac{\chi_0}{1+\chi_0} \cG_F} \times \pare{-r^\tau + \frac{1}{1+\chi_0} \cG_F}}. 
\end{multline}

We rely now on the following technical lemma whose proof is postponed:
\begin{lemma}\label{lem:Gtau_to_0}
The function $ G^\tau $ converges to zero as $ \tau\to 0 $ in $ L^2 \pare{\pare{\varepsilon, T} ; \dot{H}^{- \frac{1}{2} }} $. 
\end{lemma} 
We can hence endow the system with an appropriate initial data at time $ t=\varepsilon $ in order to deduce the following Cauchy problem satisfied by $ \delta u^\tau $:
\begin{equation}\label{eq:deltau}
\left\lbrace
\begin{aligned}
& \partial_t  \delta u^\tau  - \nu \Delta  \delta u^\tau  + \nabla\delta p^\tau = -  \delta u^\tau \cdot\nabla u^\tau + \bu\cdot\nabla  \delta u^\tau + G^\tau, & \pare{x, t}\in\bR^3\times\pare{\varepsilon, T}, \\
& \div\ \delta u^\tau =0,  & \pare{x, t}\in\bR^3\times\pare{\varepsilon, T},
\\
& \left. \delta u^\tau\right|_{t=\varepsilon}= u^{\tau}\pare{\cdot, \varepsilon} - \bu\pare{\cdot, \varepsilon}, & x\in \bR^3 .
\end{aligned}
\right.
\end{equation}
We can hence perform an $ \Hud $ energy estimate onto the system \eqref{eq:deltau} deducing the following energy inequality
\begin{equation}\label{eq:energy_ineq1}
\frac{1}{2}\frac{\d}{\d t} \norm{\delta u^\tau\pare{t}}_{\Hud}^2 + \nu \norm{\nabla\delta u^\tau\pare{t}}_{\Hud}^2 \leqslant \av{\ps{\delta u^\tau \cdot\nabla u^\tau}{\delta u^\tau}_{\Hud}} + \av{\ps{\bu\cdot\nabla  \delta u^\tau}{\delta u^\tau}_{\Hud}} + \av{\ps{G^\tau}{\delta u^\tau}_{\Hud}}. 
\end{equation}
The following bounds are moreover immediate for any $ \alpha >0 $
\begin{equation}\label{eq:bounds_energy_ineq}
\begin{aligned}
\av{\ps{G^\tau}{\delta u^\tau}_{\Hud}} & \leqslant \alpha\nu \norm{\nabla\delta u^\tau\pare{t}}_{\Hud}^2 + \frac{C}{\alpha\nu} \norm{G^\tau\pare{t}}_{\dot{H}^{-\frac{1}{2}}}^2, \\
\av{\ps{\delta u^\tau \cdot\nabla u^\tau}{\delta u^\tau}_{\Hud}} & \leqslant \alpha\nu \norm{\nabla\delta u^\tau\pare{t}}_{\Hud}^2 + \frac{C}{\alpha\nu} \norm{u^\tau\pare{t}}^4_{\dot{H}^1} \norm{\delta u^\tau\pare{t}}_{\Hud}^2, \\
\av{\ps{\bu\cdot\nabla  \delta u^\tau}{\delta u^\tau}_{\Hud}}& \leqslant \alpha\nu \norm{\nabla\delta u^\tau\pare{t}}_{\Hud}^2 + \frac{C}{\alpha\nu} \norm{\bu \pare{t}}^4_{\dot{H}^1} \norm{\delta u^\tau\pare{t}}_{\Hud}^2. 
\end{aligned}
\end{equation}
Whence selecting $ \alpha\in\pare{0, \frac{1}{8}} $,  combining the inequalities of \eqref{eq:energy_ineq1} and \eqref{eq:bounds_energy_ineq} and applying a standard Gronwall argument we deduce the following bound for any $ t\in\pare{\varepsilon, T} $
\begin{multline}\label{eq:energy_ineq2}
\norm{\delta u^\tau \pare{t}}_{\Hud}^2 + \nu \int_\varepsilon^t \norm{\nabla\delta u^\tau\pare{t'}}_{\Hud}^2
\exp\set{\int _{t'}^t \norm{u^\tau\pare{t''}}_{\Hud}^4 + \norm{\bu\pare{t''}}_{\Hud}^4 \d t''}\d t' \\
\leqslant C \eta_\varepsilon \ \exp\set{\int _{\varepsilon}^t \norm{u^\tau\pare{t'}}_{\Hud}^4 + \norm{\bu\pare{t'}}_{\Hud}^4 \d t'}\\
 + \frac{C}{\nu} \int_0^t  \norm{G^\tau\pare{t'}}_{\dot{H}^{-\frac{1}{2}}}^2 \exp\set{\int _{t'}^t \norm{u^\tau\pare{t''}}_{\Hud}^4 + \norm{\bu\pare{t''}}_{\Hud}^4 \d t''}\d t'
\end{multline}
defining hence 
\begin{equation*}
\Phi_{u^\tau, \bu }\pare{t',t}= \exp\set{\int _{t'}^t \norm{u^\tau\pare{t''}}_{\Hud}^4 + \norm{\bu\pare{t''}}_{\Hud}^4 \d t''}, 
\end{equation*}
and since $ u^\tau, \bu\in\LqHu $ we deduce that
\begin{align*}
\Phi_{u^\tau, \bu }\pare{t',t}\geqslant 1 , && \Phi_{u^\tau, \bu }\pare{t',t} \leqslant K_{u^\tau, \bu }, 
\end{align*}
whence \eqref{eq:energy_ineq2} can be rewritten in the following more compact form
\begin{equation}\label{eq:energy_ineq3}
\norm{\delta u^\tau \pare{t}}_{\Hud}^2 + \nu \int_\varepsilon^t \norm{\nabla\delta u^\tau\pare{t'}}_{\Hud}^2
\d t' \leqslant \frac{K_{u^\tau, \bu }}{\nu} \pare{\eta_\varepsilon + \norm{G^\tau}_{L^2_T\dot{H}^{-1/2}}}, 
\end{equation}
but $ \norm{G^\tau}_{L^2_T\dot{H}^{-1/2}}\xrightarrow{\tau\to 0} 0 $ thanks to the result stated in Lemma \ref{lem:Gtau_to_0}, and since $ \eta_\varepsilon\xrightarrow{\varepsilon\to 0} $ the right hand side of \eqref{eq:energy_ineq3} can be made arbitrarily small, proving hence the convergence. 
\end{enumerate}
\end{proof}

\textit{Proof of Lemma \ref{lem:Gtau_to_0}} : 
Let us remark that we can rewrite the function $ G^\tau $ as
\begin{multline*}
G^\tau = 
\pare{m^\tau+r^\tau }\cdot \nabla \pare{-r^\tau + \frac{1}{1+\chi_0} \cG_F}
-\frac{\chi_0}{1+\chi_0} {  \cG_F}\cdot \nabla {r^\tau } \\
 -\frac{1}{2}\curl \bra{\pare{m^\tau+r^\tau + \frac{\chi_0}{1+\chi_0} \cG_F} \times {r^\tau }}
 + \frac{1}{2}\curl \bra{\pare{m^\tau+r^\tau } \times \pare{-r^\tau+ \frac{1}{1+\chi_0} \cG_F }}
.
\end{multline*}
Eventually commuting derivatives on terms of the form $ \cG_F\otimes \nabla \pare{m^\tau, r^\tau} $, $ G^\tau $ can again be rewritten in the following compact form
\begin{equation*}
G^\tau = R^\tau \otimes q_1\pare{\partial} R^\tau + q_2\pare{\partial} \pare{ R^\tau \otimes R^\tau} +  R^\tau\otimes p_1\pare{\partial}\cG_F + p_2\pare{\partial}\pare{\cG_F\otimes R^\tau}, 
\end{equation*}
where $ q_1, q_2, p_1, p_2 $ are matrix-valued homogeneous Fourier multiplier of order one and $ R^\tau = m^\tau $ or $ r^\tau $. Hence using Lemma \ref{lem:Sob_product_rules} and the Sobolev interpolation inequality $ \norm{f}_{\Hu}\lesssim \norm{f}_{\Hud}^{1/2}\norm{\nabla f}_{\Hud}^{1/2} $ we deduce
\begin{align*}
\norm{R^\tau \otimes q_1\pare{\partial} R^\tau}_{L^2 \pare{\pare{\varepsilon, T} ; \dot{H}^{- \frac{1}{2} }}} & \leqslant C \norm{R^\tau}_{L^\infty \pare{\pare{\varepsilon, T} ; \dot{H}^{ \frac{1}{2} }}} \norm{\nabla R^\tau}_{L^2 \pare{\pare{\varepsilon, T} ; \Hud}}, \\
\norm{q_2\pare{\partial} \pare{ R^\tau \otimes R^\tau}}_{L^2 \pare{\pare{\varepsilon, T} ; \dot{H}^{- \frac{1}{2} }}} & \leqslant C \norm{R^\tau}_{L^\infty \pare{\pare{\varepsilon, T} ; \dot{H}^{ \frac{1}{2} }}} \norm{\nabla R^\tau}_{L^2 \pare{\pare{\varepsilon, T} ; \Hud}}, \\
\norm{R^\tau\otimes p_1\pare{\partial}\cG_F}_{L^2 \pare{\pare{\varepsilon, T} ; \dot{H}^{- \frac{1}{2} }}} & \leqslant C \norm{R^\tau}_{L^\infty \pare{\pare{\varepsilon, T} ; \dot{H}^{ \frac{1}{2} }}} \norm{\nabla \cG_F}_{L^2 \pare{\pare{\varepsilon, T} ; \Hud}}, \\
\norm{p_2\pare{\partial}\pare{\cG_F\otimes R^\tau}}_{L^2 \pare{\pare{\varepsilon, T} ; \dot{H}^{- \frac{1}{2} }}} & \leqslant C \norm{R^\tau}_{L^\infty \pare{\pare{\varepsilon, T} ; \dot{H}^{ \frac{1}{2} }}} \norm{\nabla \cG_F}_{L^2 \pare{\pare{\varepsilon, T} ; \Hud}}, 
\end{align*}
and each of the above terms converge to zero as $ \tau\to 0 $ thanks to the hypothesis assumed on $ \cG_F $, the uniform bound \eqref{eq:energy_Hud_est_mr} and the convergence result \eqref{eq:conv_mr_to_0}. 
\hfill$ \Box $

\footnotesize{
\providecommand{\bysame}{\leavevmode\hbox to3em{\hrulefill}\thinspace}
\providecommand{\MR}{\relax\ifhmode\unskip\space\fi MR }
\providecommand{\MRhref}[2]{%
  \href{http://www.ams.org/mathscinet-getitem?mr=#1}{#2}
}
\providecommand{\href}[2]{#2}

\bibliographystyle{amsplain}}

\end{document}